\documentclass{amsart}

\usepackage{amsmath,amssymb,amscd}
\usepackage[all]{xy}
\usepackage{caption,subcaption}
\usepackage{hyperref}
\input{projection.sty}
\usepackage{stmaryrd}

\newtheorem{theorem}{Theorem}[section]
\newtheorem{lemma}[theorem]{Lemma}
\newtheorem{corollary}[theorem]{Corollary}
\newtheorem{proposition}[theorem]{Proposition}

\theoremstyle{definition}
\newtheorem{definition}[theorem]{Definition}
\newtheorem{example}[theorem]{Example}

\theoremstyle{remark}
\newtheorem{remark}[theorem]{Remark}
\newtheorem{exercise}[theorem]{Exercise}
\newtheorem{conjecture}[theorem]{Conjecture}

\newtheorem{question}[theorem]{Question}

\numberwithin{equation}{section}

\newcommand{\op}[1]{\operatorname{#1}}
\newcommand{\GL}{\operatorname{GL}}
\newcommand{\SL}{\operatorname{SL}}
\newcommand{\Gr}{\operatorname{Gr}}
\newcommand{\Fl}{\operatorname{Fl}}

\newcommand{\Rep}{\operatorname{Rep}}

\newcommand{\Hom}{\operatorname{Hom}}
\newcommand{\PHom}{\operatorname{PHom}}
\newcommand{\Aut}{\operatorname{Aut}}

\newcommand{\Ext}{\operatorname{Ext}}
\newcommand{\ext}{\operatorname{ext}}

\newcommand{\e}{{\sf e}}
\newcommand{\f}{{\sf f}}
\newcommand{\g}{{\sf g}}
\newcommand{\h}{{\sf h}}

\newcommand{\ep}{{\epsilon}}
\newcommand{\vep}{{\varepsilon}}
\newcommand{\Tr}{\operatorname{Tr}}
\newcommand{\Ker}{\operatorname{Ker}}
\newcommand{\Coker}{\operatorname{Coker}}
\newcommand{\Img}{\operatorname{Im}}
\newcommand{\QF}{\op{QF}}
\newcommand{\T}{\operatorname{T}}
\newcommand{\SI}{\operatorname{SI}}

\newcommand{\res}{\operatorname{res}}
\newcommand{\ind}{\operatorname{ind}}
\newcommand{\proj}{\operatorname{proj}\text{-}}

\newcommand{\ckQ}{\widehat{k\Delta}}

\renewcommand{\hat}[1]{\widehat{#1}}
\renewcommand{\ll}{\text{\rm \l}}
\newcommand{\pp}{\text{\rm p}}
\renewcommand{\S}{{\bf S}}

\newcommand{\mc}[1]{\mathcal{#1}}
\newcommand{\mb}[1]{\mathbb{#1}}

\newcommand{\bs}[1]{\boldsymbol{#1}}
\renewcommand{\b}[1]{\bold{#1}}
\newcommand{\cone}{\mb{R}_+\Sigma_{\beta}(Q)}
\newcommand{\br}[1]{\overline{#1}}
\newcommand{\innerprod}[1]{\langle#1\rangle}
\newcommand{\dimbar}{\underline{\dim}}

\newcommand{\sm}[1]{\left(\begin{smallmatrix}#1\end{smallmatrix}\right)}

\newcommand{\fr}[1]{\framebox[1.2\width]{{$#1$}}}

\renewcommand{\L}{{\mb{L}}}
\newcommand{\M}{\mb{M}}

\newcommand{\uca}{\br{\mc{C}}}
\newcommand{\Deltauv}{\Delta_{\b{v}}^{\b{u}}}

\begin{document}

\title{Cluster Algebras and Semi-invariant Rings II. Projections}
\author{Jiarui Fei}
\address{National Center for Theoretical Sciences, Taipei 30043, Taiwan}
\email{jrfei@ncts.ntu.edu.tw}
\thanks{}


\date{}

\begin{abstract} Let $\SI_\beta(Q)$ be the semi-invariant ring of $\beta$-dimensional representations of a quiver $Q$.
Suppose that $(Q,\beta)$ projects to another quiver with dimension vector $(Q',\beta')$ through an exceptional representation $E$.
We show that if $\SI_\beta(Q)$ is the upper cluster algebra associated to an ice quiver $\Delta$,
then $\SI_{\beta'}(Q')$ is the upper cluster algebra associated to $\Delta'$,
where $\Delta'$ is obtained from $\Delta$ through simple operations depending on $E$.
We also study the relation of their bases using the quiver with potential model.
\end{abstract}

\maketitle

\section*{Introduction}

The {\em cluster algebra} of Fomin-Zelevinsky was turned out be ubiquitous in algebraic Lie theory and classical invariant theory.
It had been realized under different motivation (eg., \cite{FG,Fk1}) that the {\em upper cluster algebra} \cite{BFZ} is a more useful notion than the cluster algebra.
The purpose of this series of papers is to understand when a semi-invariant of quiver representations is an upper cluster algebra, and use the cluster structure solve problems in invariant theory.

One family of such upper cluster algebras found in \cite{Fs1} is the semi-invariant ring of complete {\em triple flags} introduced in \cite{DW1}.
They are {\em multigraded} upper cluster algebras of the {\em hive quivers}. The grading is also very useful because the dimension of each graded piece counts certain {\em Littlewood-Richardson coefficient}.
On the other hand, many quivers can be projected from the triple flag quivers through exceptional sequences.
In fact, the author does not know any quiver without oriented cycles that cannot be projected from a triple flag quiver.
As we shall see in this paper that finding the cluster structure of semi-invariant ring of their representations will be ultimately reduced to the case of complete triple flags.

More naturally and generally, one can ask how a graded upper cluster algebra (or GUCA in short) might change into another one under degree (or weight) restriction.
We are unable to answer this question in general.
However, if the grading is restricted to the hyperplane $\theta^\perp$ cut out by some weight $\theta$,
then we can provide a complete solution under some additional conditions.

There is one special case which is particularly interesting to us.
That is when $\theta=\ep$ is a {\em projector} in the graded upper cluster algebra $\mc{A}$.
Here, being a projector simply means that $\ep$ is {\em real} and $(\mc{A}_e)_{\ep^\perp}=\mc{A}_{\ep^\perp}$ for some $e\in \mc{A}_{\ep}$.
In this case, the restriction leads to deleting some vertex.
For the general cases, we need more restriction on the weights $\bs{\sigma}$ of an extended cluster.
In general, we need to delete a set of frozen vertices, then freeze another set of vertices.
Let $\Delta$ be an ice quiver, and $\uca(\Delta)$ be its associated upper cluster algebra.
Let $\b{v}$ be a set of vertices of $\Delta$.
We use the notation $\Delta_{\b{v}}$ and $\Delta^{\b{v}}$ for the quiver obtained from $\Delta$ by deleting and freezing $\b{v}$ respectively.

Let $Q=(Q_0,Q_1)$ be a finite quiver without oriented cycles, and $\beta$ be a fixed dimension vector of $Q$.
The space $\Rep_\beta(Q)$ of all $\beta$-dimensional representations is acted by
the product of special linear groups $\SL_\beta:=\prod_{v\in Q_0}\SL_{\beta(v)}$ by the natural base change.
The rings of semi-invariants $\SI_\beta(Q)$ is by definition the invariant ring $k[\Rep_\beta(Q)]^{\SL_\beta}$.
This algebra is graded by a sublattice of $\mb{Z}^{Q_0}$ equipped with a bilinear form $\innerprod{-,-}$.
We denote by $\SI_\beta(Q)_\sigma$ the weight-$\sigma$ component of $\SI_\beta(Q)$.

The weight restriction to $\ep^\perp:=\{\sigma\mid \innerprod{\ep,\sigma}=0\}$ applied to $\SI_\beta(Q)$ leads to a change of the quiver $Q$ and the dimension vector $\beta$.
This was studied by A. Schofield \cite{S1} and Derksen-Weyman \cite{DW2}.
To be more precise, we have that $\bigoplus_{\sigma\in\ep^\perp} \SI_\beta(Q)_{\sigma} \cong \SI_{\beta_\ep}(Q_\ep)$,
where $Q_{\ep}$ is another quiver with one vertex less than $Q$, and $\beta_\ep$ is a dimension vector of $Q_\ep$.
The general weight restriction does not necessarily lead to such a change.
However, it does in another special case called {\em vertex removal} with an additional condition \eqref{eq:beta_r} on $\beta$.
We denote the new quiver with dimension vector in this situation by $(\br{Q},\br{\beta})$.
Put all these together, and we obtain our first two main results.

\begin{theorem}[Theorem \ref{T:invGUCA}] Suppose that $\SI_\beta(Q)$ is a naturally graded UCA $\uca(\Delta,\b{x}; \bs{\sigma})$, and $B(\Delta)$ has full rank.
If $\bs{\sigma}(e)=\ep$ is real and extremal in the grading cone of $\SI_\beta(Q)$, then
we have that $\SI_{\beta_\ep}(Q_\ep)\cong \uca(\Delta_e,\b{x}_e;\bs{\sigma}_e)$.
\end{theorem}

\begin{theorem}[Theorem \ref{T:removal}] Suppose that $\SI_\beta(Q)$ is a naturally graded UCA $\uca(\Delta;\bs{\sigma})$, and $B(\Delta)$ has full rank.
Let $r$ be a vertex of $Q$ satisfying \eqref{eq:beta_r}
and that $\bs{\sigma}(v)(r) = 0$ for all $v\in \Delta_0$ except for a set of frozen vertices $\b{v}$, where all $\bs{\sigma}(v)(r)<0$ (or all $>0$).
Then we have that $\SI_{\br{\beta}}(\br{Q})\cong \uca(\Deltauv;\bs{\sigma}(\hat{\b{v}}))$, where $\b{u}$ is {\em attached} to $\b{v}$.
\end{theorem}

In \cite{DWZ1} and \cite{DWZ2}, the {\em mutation} of {\em quivers with potentials} is invented to model the cluster algebras.
For an ice quiver with potential $(\Delta,W)$, we can associated to it the {\em Jacobian algebra} $J(\Delta,W)$.
Let $G(\Delta,W)$ be the set of all {\em $\mu$-supported vectors} in $K_0(\proj J(\Delta,W))$ introduced in \cite{Fs1}.
We say an IQP $(\Delta,W)$ {\em models} an algebra $\mc{A}$ if the {\em generic cluster character} of \cite{P} maps $G(\Delta,W)$ onto a basis of $\mc{A}$.
Our other two main results say that the model is in some sense hereditary under the two operations considered before.

\begin{theorem}[Theorem \ref{T:frozen_model}] Suppose that $e$ is a projector in $\uca(\Delta;\bs{\sigma})$. If $(\Delta,W)$ models $\uca(\Delta;\bs{\sigma})$, then $(\Delta_e,W_e)$ models $\uca(\Delta_e;\bs{\sigma}_e)$.
\end{theorem}

\begin{theorem}[Theorem \ref{T:hat}] Under the same assumption as Proposition \ref{P:outerproj},
we suppose that $(\Delta,W)$ models $\uca(\Delta;\bs{\sigma})$.
Then $(\Deltauv,W_{\b{v}})$ models $\uca(\Deltauv;\bs{\sigma}(\hat{\b{v}}))$.
\end{theorem}

\noindent If in addition $G(\Delta,W)$ is given by lattice points in some polyhedron, then so are $G(\Delta_e,W_e)$ and $G(\Deltauv,W_{\b{v}})$.

A comprehensive application of this theory has been illustrated in \cite{Fk1}.
To invite readers, we still provide many other examples.
The first class of examples explains how we get some known cluster structure from \cite{Fs1}.
They include the total coordinate rings of partial flag varieties and unipotent subgroups of $\GL_n$ \cite{GLSa}.
The second class of examples is mildly new.
They are the semi-invariant rings of complete $m$-tuple flags for $m\leq 6$, whose cluster structure can also be obtained from \cite{Fs1}.

\subsection*{Outline of the paper}
In Section \ref{S:GUCA} we recall the graded upper cluster algebra following \cite{BFZ,FZ4,FP}.
The new emphasis is on the grading given by a weight configuration of a unimodular lattice.
In Section \ref{S:Proj}, we introduce the crucial map $\ell_e$ and the definition of a projector.
We prove the precursors of the first two main results -- Proposition \ref{P:subalgebra} and \ref{P:outerproj}.
In Section \ref{S:SI}, we recall the theory of semi-invariant ring of quiver representations with an emphasis on the two operations -- orthogonal projection and vertex removal. We give a characterization of a projector in a semi-invariant ring in Proposition \ref{P:extremal_proj}.
We prove our first two main results -- Theorem \ref{T:invGUCA} and \ref{T:removal}, and provide many examples.
In Section \ref{S:CC}, we recall the ice quivers with potentials and their representations following \cite{DWZ1}.
We recall the generic cluster character in the current setting following \cite{P,Fs1}.
In Section \ref{S:model}, we use the restriction and induction functor to study the relation of the cluster models under the two operations.
We prove the other two main results -- Theorem \ref{T:frozen_model} and \ref{T:hat}, and provide examples.

\subsection*{Notations and Conventions}
Our vectors are exclusively row vectors. All modules are right modules.
For a quiver $Q$, we denote by $Q_0$ the set of vertices and by $Q_1$ the set of arrows.
For an arrow $a$, we denote by $t(a)$ and $h(a)$ its tail and head.
Arrows are composed from left to right, i.e., $ab$ is the path $\cdot \xrightarrow{a}\cdot \xrightarrow{b} \cdot$.
Unadorned $\Hom, \otimes$ and $\dim$ are all over the base field $k$. 
We write $\hom$ for $\dim \Hom$.
For any representation $M$, $\dimbar M$ is the dimension vector of $M$.
For direct sum of $n$ copies of $M$, we write $nM$ instead of the traditional $M^{\oplus n}$.

\section{Graded Upper Cluster Algebras} \label{S:GUCA}
\subsection{Cluster Algebras}
A quiver is just a directed graph with possibly multiple arrows between two vertices.
Throughout we assume that quivers have no loops or oriented 2-cycles.
An {\em ice} quiver $\Delta$ is a quiver, where some vertices are designated as {\em mutable} while the rest are designated {\em frozen}.
We usually label the quiver such that the first $p$ vertices are mutable.
If we require no arrows between frozen vertices, then
such a quiver is uniquely determined by its {\em $B$-matrix} $B(\Delta)$.
It is a $p\times q$ matrix given by
$$b_{u,v} = |\text{arrows }u\to v| - |\text{arrows }v \to u|.$$
The combinatorial data defining a cluster algebra is completely encoded in an ice quiver.

\begin{definition} \label{D:Qmu}
Let $u$ be a mutable vertex of $\Delta$.
The {\em quiver mutation} $\mu_u$ transforms $\Delta$ into the new quiver $\Delta'=\mu_u(\Delta)$ via a sequence of three steps.
\begin{enumerate}
\item For each pair of arrows $v\to u\to w$, introduce a new arrow $v\to w$ (unless both $v$ and $w$ are frozen, in which case do nothing);
\item Reverse the direction of all arrows incident to $u$;
\item Remove all oriented 2-cycles.
\end{enumerate}
\end{definition}

\noindent The above recipe can be reformulated in terms of $B$-matrix as follows.
Let $\phi$ be the $q\times q$ matrix obtained from the identity matrix by replacing the $u$-th row by a vector $\phi_u$ where $$\phi_u(u)=-1,\quad \text{and}\quad \phi_u(v)=|\text{arrows } v\to u |.$$
We write $\phi_{\op{p}}$ for the restriction of $\phi$ to its $p\times p$ upper left corner.
Then the mutated $B$-matrix $B'$ for $\Delta'$ is related to the original one by
\begin{equation} \label{eq:mu_B} B'=\phi_{\op{p}}^{\T} B\phi.
\end{equation}
We note that $\phi=\phi^{-1}$ and $\phi_{\op{p}}=\phi_{\op{p}}^{-1}$
so the quiver mutation is an {\em involution} and preserves the rank of the $B$-matrix.

\begin{definition} \label{D:seeds} 
Let $\mc{F}$ be a field containing $k$.
A {\em seed} in $\mc{F}$ is a pair $(\Delta,\b{x})$ consisting of an ice quiver $\Delta$ together with a collection $\b{x}=\{x_1,x_2,\dots,x_q\}$, called an {\em extended cluster}, consisting of algebraically independent (over $k$) elements of $\mc{F}$, one for each vertex of $\Delta$.
The elements of $\b{x}$ associated with the mutable vertices are called {\em cluster variables}; they form a {\em cluster}.
The elements associated with the frozen vertices are called
{\em frozen variables}, or {\em coefficient variables}.

A {\em seed mutation} $\mu_u$ at a (mutable) vertex $u$ transforms $(\Delta,\b{x})$ into the seed $(\Delta',\b{x}')=\mu_u(\Delta,\b{x})$ defined as follows.
The new quiver is $\Delta'=\mu_u(\Delta)$.
The new extended cluster is
$\b{x}'=\b{x}\cup\{x_{u}'\}\setminus\{x_u\}$
where the new cluster variable $x_u'$ replacing $x_u$ is determined by the {\em exchange relation}
\begin{equation*} \label{eq:exrel}
x_u\,x_u' = \prod_{v\rightarrow u} x_v + \prod_{u\rightarrow w} x_w.
\end{equation*}
\end{definition}

\noindent We note that the mutated seed $(\Delta',\b{x}')$ contains the same
coefficient variables as the original seed $(\Delta,\b{x})$.
It is easy to check that one can recover $(\Delta,\b{x})$
from $(\Delta',\b{x}')$ by performing a seed mutation again at $u$.
Two seeds $(\Delta,\b{x})$ and $(\Delta^\dag,\b{x}^\dag)$ that can be obtained from each other by a sequence of mutations are called {\em mutation-equivalent}, denoted by $(\Delta,\b{x})\sim (\Delta^\dag,\b{x}^\dag)$.
If $\Delta$ and $\Delta^\dag$ are clear from the context, we may just write $\b{x} \sim \b{x}^\dag$.

\begin{definition}[{Cluster algebra}] \label{D:cluster-algebra}
The {\em cluster algebra $\mc{C}(\Delta,\b{x})$} associated to a seed $(\Delta,\b{x})$ is defined as the subring of $\mc{F}$
generated by all elements of all extended clusters of the seeds mutation-equivalent to $(\Delta,\b{x})$.
\end{definition}

\noindent Note that the above construction of $\mc{C}(\Delta,\b{x})$ depends only, up to a natural isomorphism, on the mutation equivalence class of the initial quiver $\Delta$. So we may drop $\b{x}$ and simply write $\mc{C}(\Delta)$.

\subsection{Upper Cluster Algebras}

An amazing property of cluster algebras is
\begin{theorem}[{Laurent Phenomenon}, \textrm{\cite{FZ1,BFZ}}] \label{T:Laurent}
Any element of a cluster algebra $\mc{C}(\Delta,\b{x})$ can be expressed in terms of the
extended cluster $\b{x}$ as a Laurent polynomial, which is polynomial in coefficient variables.
\end{theorem}

\noindent Since $\mc{C}(\Delta,\b{x})$ is generated by cluster variables from the seeds mutation equivalent to $(\Delta,\b{x})$,
Theorem \ref{T:Laurent} can be rephrased as
$$\mc{C}(\Delta,\b{x}) \subseteq \bigcap_{(\Delta^\dag,\b{x}^\dag) \sim (\Delta,\b{x})}\mc{L}_{\b{x}^\dag},$$
where $\mc{L}_{\b{x}}:=k[x_1^{\pm 1},\dots,x_p^{\pm 1}, x_{p+1}, \dots x_{q}]$.
Note that our definition of $\mc{L}_{\b{x}}$ is slightly different from the original one in \cite{BFZ}, where the coefficient variables are inverted in $\mc{L}_{\b{x}}:=k[x_1^{\pm 1},\dots,x_p^{\pm 1}, x_{p+1}^{\pm 1}, \dots x_{q}^{\pm 1}]$.
\begin{definition}[{Upper Cluster Algebra}]
The {\em upper cluster algebra} (or UCA in short) with seed $(\Delta,\b{x})$ is
$$\uca(\Delta,\b{x}):=\bigcap_{(\Delta^\dag,\b{x}^\dag) \sim (\Delta,\b{x})}\mc{L}_{\b{x}^\dag}.$$
\end{definition}

In general, there may be infinitely many seeds mutation equivalent to $(\Delta,\b{x})$.
So the above definition is not very useful to test the membership in an upper cluster algebra.
However, the following theorem allows us to do only finitely many checking.

\begin{definition}\cite{BFZ}
We define the upper bounds
$$\mc{U}(\Delta,\b{x}):=\bigcap_{1\leq u\leq p}\mc{L}_{\mu_u(\b{x})}.$$
\end{definition}

\begin{theorem}\cite[Corollary 1.9]{BFZ} \label{T:bounds} Suppose that $B(\Delta)$ has full rank, and $(\Delta,\b{x})\sim (\Delta^\dag,\b{x}^\dag)$,
then $\mc{U}(\Delta,\b{x})=\mc{U}(\Delta^\dag,\b{x}^\dag)$. In particular, $\mc{U}(\Delta,\b{x})=\uca(\Delta,\b{x})$.
\end{theorem}

\begin{remark} This theorem is originally proved for $\mc{U}(\Delta,\b{x})$ and $\uca(\Delta,\b{x})$ with $\mc{L}_{\b{x}}$ defined there.
However, if we carefully examine the argument, we find that the result is also valid for our $\mc{L}_{\b{x}}$.
\end{remark}

Any upper cluster algebra, being a subring of a field, is an integral domain (and under our conventions, a $k$-algebra).
However, it may fail to be Noetherian \cite{Sp}.
Since normality is preserved under localization and intersection, any UCA is normal.
The next lemma is useful to identify an UCA as a subalgebra of some given Noetherian normal domain.

Let $\b{x}$ be an extended cluster in some UCA.
If $\b{x}$ is contained in another commutative ring $R$, we write $Z_R(\b{x})$ for the union of the zero locus in $\op{Spec}(R)$ defined by all cluster variables in $\b{x}$.
\begin{definition}  \label{D:CR1} We say that a seed $(\Delta,\b{x})$ is cluster-regular in codimension 1 (or CR1 in short) in $R$ if each adjacent cluster $\b{x}'\subset R$, and
$Z_R(\b{x}) \cap \left(\bigcap_{\b{x}'\stackrel{\mu_u}{\sim} \b{x}}Z_R(\b{x}') \right)$ has codimension $\geq 2$ in $\op{Spec} R$.
\end{definition}
\noindent We call two elements of $R$ {\em coprime in codimension 1} if the locus of their common zeros
has codimension~$\ge 2$ in $\operatorname{Spec}(R)$.
By expanding $Z_R(\b{x}) \cap \left(\bigcap_{\b{x}'\stackrel{\mu_u}{\sim} \b{x}}Z_R(\b{x}') \right)$,
we can easily see that the codimension condition in Definition \ref{D:CR1} is equivalent to that
\begin{align}\label{eq:CR1}
\text{Each pair of cluster variables in } \b{x} & \text{ and each pair } (x_u,x_u') \notag \\
& \text{ are coprime in codimension 1 in } R.
\end{align}

\begin{lemma}[{\cite[Proposition 3.6]{FP}}] \label{L:RCA}
Let $R$ be a $k$-algebra and a Noetherian normal domain.
If $(\Delta,\b{x})$ is a CR1 seed in $R$, then $R\supseteq\uca(\Delta,\b{x})$.
\end{lemma}

\begin{proof}
For any $r$ in the complement of $Z_R(\b{x}) \cap \left(\bigcap_{\b{x}'\stackrel{\mu_u}{\sim} \b{x}}Z_R(\b{x}') \right)$ in $\op{Spec} R$,
we can find a seed $(\Delta^\circ,\b{x}^\circ)$ (either the original one or an adjacent one) such that any cluster variable in $\b{x}^\circ$ does not vanish at $r$. We can express any $z\in \uca(\Delta,\b{x})$ as a Laurent polynomial in $\b{x}^\circ$, which is polynomial in coefficient variables.
Since $\b{x}^\circ \subset R$, we conclude that $z$ is regular at $r$.
We know from \cite[Corollary 5.22]{AM} that if $R$ is a Noetherian domain, then $\bigcap_{\op{ht} P=1}R_P$ is the integral closure of $R$.
By the CR1 assumption, we conclude that $z$ is in the integral closure of $R$ which is $R$ itself by the normality.
\end{proof}

\subsection{$\g$-vectors and Gradings} \label{s:gv}
Let $\b{x}=\{x_1,x_2,\dots,x_q\}$ be an (extended) cluster.
For a vector $\g\in \mb{Z}^q$, we write $\b{x}^\g$ for the monomial $x_1^{\g(1)}x_2^{\g(2)}\cdots x_q^{\g(q)}$.
For $u=1,2,\dots,p$, we set ${y}_u= \b{x}^{-b_{u}}$ where $b_u$ is the $u$-th row of the matrix $B(\Delta)$,
and let ${\b{y}}=\{{y}_1,{y}_2,\dots,{y}_p\}$.

Suppose that an element $z\in\uca(\Delta)$ can be written as
\begin{equation}\label{eq:z} z = \b{x}^{\g(z)} F({y}_1,{y}_2,\dots,{y}_p),
\end{equation}
where $F$ is a rational polynomial not divisible by any ${y}_i$, and $\g(z)\in \mb{Z}^q$.
If we assume that the matrix $B(\Delta)$ has full rank,
then the elements ${y}_1,{y}_2,\dots,{y}_p$ are algebraically independent so that the vector $\g(z)$ is uniquely determined \cite{FZ4}.
We call the vector~$\g(z)$~the (extended) {\em $\g$-vector} of $z$.
Definition implies at once that for two such elements $z_1,z_2$ we have that
$\g(z_1z_2) = \g(z_1) + \g(z_2)$.
So the set $G(\Delta)$ of all $\g$-vectors in $\uca(\Delta)$ forms a sub-semigroup of $\mb{Z}^q$.

\begin{lemma}[{\cite[Lemma 5.5]{Fs1}, {\em cf.} \cite{P}}] \label{L:independent} Assume that the matrix $B(\Delta)$ has full rank.
Let $Z=\{z_1,z_2,\dots,z_k\}$ be a subset of $\uca(\Delta)$ with distinct well-defined $\g$-vectors.
Then $Z$ is linearly independent over $k$.
\end{lemma}

Let $\L\cong \mb{Z}^n$ be a {\em unimodular} lattice equipped with a bilinear form $\L\times \L \to \mb{Z}$, $\alpha,\beta \mapsto \innerprod{\alpha,\beta}$.
The bilinear form may be neither symmetric nor skew-symmetric but must be unimodular, i.e., the mapping
$$\L \xrightarrow{\alpha \mapsto \innerprod{\alpha,-}} \Hom_{\mb{Z}}(\L,\mb{Z})$$
is a $\mb{Z}$-module isomorphism.
This kind of lattice is the one discussed throughout this paper.

For any $\theta\in \L$, we write
$$^\perp \theta := \{\alpha\in \L \mid \innerprod{\alpha,\theta}=0\},\text{ \ and \ } \theta^\perp := \{\beta\in \L \mid \innerprod{\theta,\beta}=0\}$$
for its left and right {\em orthogonal sublattices}.
The {\em Coxeter transformation} $\tau$ is the unique linear transformation determined by $\innerprod{\alpha,\beta}=-\innerprod{\beta,\tau\alpha}$.
Note that $^\perp\theta\cong\theta^\perp$ via the Coxeter transformation, allowing us to restrict our discussion to right orthogonal sublattices.
For a {\em real} vector $\ep\in\mb{L}$, i.e., $\innerprod{\ep,\ep}=1$, we define two maps \begin{align*}
r_\ep: \L\to {^\perp\ep},\ \alpha\mapsto \alpha-\innerprod{\alpha,\ep}\ep, \text{ \ and \ } l_\ep: \L\to \ep^\perp,\ \beta\mapsto \beta-\innerprod{\ep,\beta}\ep.
\end{align*}
called the right and left {\em projection} through $\ep$.

\begin{definition} \label{D:wtconfig} A {\em weight configuration} $\bs{\sigma}$ of $\mb{L}$ on an ice quiver $\Delta$ is an assignment for each vertex $v$ of $\Delta$ a (weight) vector $\bs{\sigma}(v)\in \mb{L}$ such that for each mutable vertex $u$, we have that
\begin{equation} \label{eq:weightconfig}
\sum_{v\to u} \bs{\sigma}(v) = \sum_{u\to w} \bs{\sigma}(w).
\end{equation}
The {\em mutation} $\mu_u$ also transforms $\bs{\sigma}$ into a weight configuration $\bs{\sigma}'$ on the mutated quiver $\mu_u(\Delta)$ defined as
\begin{equation*} \label{eq:mu_wt}
\bs{\sigma}'(v) = \begin{cases} \displaystyle \sum_{u\to w} \bs{\sigma}(w) - \bs{\sigma}(u) & \text{if } v=u, \\ \bs{\sigma}(v) & \text{if } v\neq u. \end{cases}\end{equation*}
\end{definition}

\noindent We usually write such a weight configuration as a pair $(\Delta,\bs{\sigma})$ and its mutation by $\mu_u(\Delta,\bs{\sigma})$.
By slight abuse of notation, we can view $\bs{\sigma}$ as a matrix whose $v$-th row is the weight vector $\bs{\sigma}(v)$.
In this matrix notation, the condition \eqref{eq:weightconfig} is equivalent to $B\bs{\sigma}$ is a zero matrix.
For any weight configuration of $\Delta$, the mutation can be iterated because of \eqref{eq:mu_B}.

Let $(\Delta,\bs{\sigma})$ be a weight configuration of a unimodular lattice.
Let $e$ be a vertex of $\Delta$ such that $\bs{\sigma}(e)$ corresponds to a real vector $\ep$.
We define the (left) {\em projection} of the weight configuration $(\Delta,\bs{\sigma})$ through $e$ as the pair $(\Delta_e,\bs{\sigma}_e)$,
where $\Delta_e$ is the full subquiver obtained from $\Delta$ by deleting the vertex $e$, and
$$\bs{\sigma}_e(v)= l_\ep(\bs{\sigma}(v)) =\bs{\sigma}(v)-\left\langle\ep,\bs{\sigma}(v) \right\rangle \ep.$$
It is easy to see that the projection commutes with mutations away from $e$.

Given a weight configuration $(\Delta,\bs{\sigma})$,
we can assign a multidegree (or weight) to the upper cluster algebra $\uca(\Delta,\b{x})$ by setting
$\deg(x_v)=\bs{\sigma}(v)$ for $v=1,2,\dots,q$.
Then mutation preserves multihomogeneousity.
We say that this upper cluster algebra is $\bs{\sigma}$-graded, and denoted by $\uca(\Delta,\b{x};\bs{\sigma})$.
We refer to $(\Delta,\b{x};\bs{\sigma})$ as a graded seed.
Note that the variables in $\b{y}$ have zero degrees.
So if $z$ has a well-defined $\g$-vector as in \eqref{eq:z}, then $z$ is homogeneous of degree $\g\bs{\sigma}$.



\section{Projections from Graded Cluster Algebras} \label{S:Proj}

\subsection{The map $\ell_e$}

Let $\M$ be a monoid. For simplicity, we assume that $\M$ is a submonoid inside a unimdoular lattice $\L:=\mb{Z}^n$.
Let $A$ be a $\M$-graded $k$-algebra, which is also an integral domain.  
We do not assume that $A$ is Noetherian, or have finite-dimensional graded pieces.
We will speak of a $\M$-graded $k$-algebra just as a graded algebra.
For any $\ep\in \L$, we define $A_{\ep^\perp}$ to be the graded subalgebra $\bigoplus_{\alpha\in \M\mid \ep\perp\alpha} A_{\alpha}$ of $A$.
We denote by $A_e$ the localization of $A$ at $e\in A$.

For any $e\in A$ of real weight $\ep$, we define the algebra homomorphism $\ell_e: A\to (A_e)_{\ep^\perp}$ by the following $k$-linear extension on the homogeneous components:
$$\forall a\in A_\alpha,\ a \mapsto a e^{-\innerprod{\ep,\alpha}}.$$
Clearly $\ell_e$ fixes $A_{\ep^\perp}$.
For any $ae^{-n}\in (A_e)_{\ep^\perp}$, we have that $\ell_e(a)=ae^{-n}$ so $\ell_e$ is surjective.
The homomorphism $\ell_e$ can be uniquely extended to the quotient fields $\QF(A)\to \QF(A)_{\ep^\perp}$, which fixes $\QF(A)_{\ep^\perp}$.
We denote the extension still by $\ell_e$.
Let $i_e$ be the embedding $\QF(A)_{\ep^\perp} \hookrightarrow \QF(A)$.

\begin{definition} \label{D:projector} An element $e\in A_\ep$ is called a {\em projector} in $A$ if $\ep$ is real and $(A_e)_{\ep^\perp}=A_{\ep^\perp}$.
Such an $\ep\in\M$ is also called a {\em projector} for $A$.
\end{definition}

\noindent If $e$ is a projector, then $\ell_e$ maps $A$ onto $A_{\ep^\perp}$. We call $\ell_e$ the {\em projection through} $e$.

\begin{lemma} \label{L:projector} $e\in A_\ep$ is a projector if and only if for any $f\in A_\alpha$ with $\innerprod{\ep,\alpha}=d>0$, $f$ factors as $f=e^d f'$ for some $f'\in A_{\ep^\perp}$.
\end{lemma}

\begin{proof}  Clearly $(A_e)_{\ep^\perp}\supseteq A_{\ep^\perp}$. For $\forall f\in A_\alpha$ such that $f e^{-d} \in (A_e)_{\ep^\perp}$,
we have that $\innerprod{\ep,\alpha}=d \innerprod{\ep,\ep}$.
\hbox{Then $(A_e)_{\ep^\perp}=A_{\ep^\perp}$ if and only if $\innerprod{\ep,\ep}=1$ and $f':=f e^{-d} \in A_{\ep^\perp}$.}
\end{proof}

\begin{proposition} \label{P:projector} Suppose that $\ep$ is a projector for $A$, then we have the following. \begin{enumerate}
\item If $A_{\bs{0}}=k$, then $\dim A_{n\ep}=1$ for any $n\in\mb{N}$.
\item If any $u\in A_{\b{0}}$ is a unit in $A$, then any $f \in A_\ep$ is irreducible in $A$.
\item If $A$ is a UFD and $A_{\bs{0}}=k$ contains all units in $A$, then $\ep$ is extremal in the cone $\mb{R}^+\M$.
\end{enumerate} \end{proposition}
\begin{proof}
(1). If $e\in A_\ep$ is a projector, then for all $f\in A_{n\ep}$ we have that $f=e^n f'$ for some $f'\in A_{\bs{0}}$. But $A_{\bs{0}}=k$, so $\dim A_{n\ep}=1$.

(2). Let $e$ be a projector in $A_\ep$.
Suppose that $f=gh\in A_\ep$ with $g\in A_\alpha$ and $h\in A_\beta$.
One of $\alpha$ and $\beta$, say $\alpha$, must have that $\innerprod{\ep,\alpha}=d>0$.
Then $g=g'e^d$ for some $g'\in A_{\ep^\perp}$.
On the other hand, $f=eu$ for some $u\in A_{\b{0}}$.
Then $eu=g'e^dh \Rightarrow g'e^{d-1}h=u$ because $A$ is a domain.
By assumption $u$ is a unit, then so is $h$.

(3). Suppose that $\ep$ is not extremal, then there is some $\alpha,\beta\in \M \setminus \mb{R}\ep$ such that $n\ep = \alpha+\beta$.
By (1), we have that $e^n = gh$ for $e\in A_{\ep}$, and some $g\in A_\alpha, h\in A_\beta$.
By (2) $e$ is irreducible, so $e^n = gh$ contradicts the UFD property of $A$.
\end{proof}

\subsection{The Cases when $\ep\in\mb{M}$}
Throughout this paper, we will use $e$ for both the vertex $e$ and the cluster variable $\b{x}(e)$.
We find no potential confusion from this abuse of notation.
From now on, the map $\ell_e$ is always defined on the quotient field $k(\b{x})$ of $\uca(\Delta,\b{x};\bs{\sigma})$.
If $\b{x}$ is an extended cluster, when we say ``Laurent polynomial in $\b{x}$", it always implies that it is polynomial in coefficient variables in $\b{x}$.
We omit this part if it is not important in the context.

The goal of this subsection is to relate $\left(\uca(\Delta;\bs{\sigma})_e\right)_{\ep^\perp}$ to another upper cluster algebra, where $\ep=\bs{\sigma}(e)$ is real.
Our first lemma reduces the general case to the case when $e$ is frozen.
By freezing a set of vertex $\b{u}$, we mean that we make $\b{u}$ frozen.
Let $\Delta^{\b{u}}$ be the quiver obtained from $\Delta$ by freezing $\b{u}$.

%

\begin{lemma} \label{L:reduce2frozen} Suppose that $B(\Delta)$ is of full rank. Let $\mc{A}:=\uca(\Delta,\b{x};\bs{\sigma})$ and $\mc{S}:=\uca(\Delta^e,\b{x}; \bs{\sigma})$, then we have that $\mc{A}_e=\mc{S}_e$.
\end{lemma}

\begin{proof} We first show that $\mc{S}$ is a subalgebra of $\mc{A}$.
By Theorem \ref{T:bounds}, it suffices to show that any $s\in\mc{S}$ can be written as a Laurent polynomial in $\b{x}':=\mu_e(\b{x})$.
Note that $s$ can be written as a Laurent polynomial $p$ in $\b{x}$, which is polynomial in $e$.
But $ee' = h$, where $h$ is the exchange polynomial. Substituting $e$ by $h(e')^{-1}$, we get the required Laurent polynomial in $\b{x}'$.

Since any cluster $\b{x}$ of $\mc{S}$ is also a cluster of $\mc{A}$,
any element in $\mc{A}$ can be written as a Laurent polynomial in $\b{x}$ but not necessary polynomial in $e$.
So after inverting $e$ in $\mc{S}$, we obtain that $\mc{S}_e \supseteq \mc{A}$.
We conclude that $\mc{S}_e = \mc{A}_e$.
\end{proof}

\begin{remark} \label{R:CR1} The lemma still holds if the full rank condition is replaced by the condition that $\mc{A}$ is Noetherian and $(\Delta,\b{x})$ is a CR1 seed in $\mc{A}$.
Indeed, if $(\Delta,\b{x})$ is a CR1 seed in $\mc{A}$, then $(\Delta^e,\b{x})$ is also a CR1 seed in $\mc{A}$ by \eqref{eq:CR1}.
By Lemma \ref{L:RCA}, we still have that $\mc{S} \subset \mc{A}$.
The same remark applies to several results depending on Lemma \ref{L:reduce2frozen}, eg.,
Proposition \ref{P:subalgebra}, Corollary \ref{C:Gproj_e}, Theorem \ref{T:invGUCA}, and Corollary \ref{C:model_proj}.
\end{remark}

For a graded seed $(\Delta,\b{x};\bs{\sigma})$, we define $\b{x}_e$ by $\b{x}_e(v)=\ell_e(\b{x}(v))=\b{x}(e)^{-\innerprod{\ep,\bs{\sigma}(v)}}\b{x}(v)$ for $v\neq e$.
Since $\b{x}$ is algebraically independent, so is $\b{x}_e$.
We have seen that $(\Delta_e;\bs{\sigma}_e)$ is a weight configuration, so $(\Delta_e, \b{x}_e ;\bs{\sigma}_e)$ is a graded seed.

\begin{lemma} \label{L:projmu} For any sequence of mutations $\bs{\mu}:=\mu_{u_k}\cdots\mu_{u_2}\mu_{u_1}$ with $u_i\neq e$, we have that
$\bs{\mu}(\Delta)_e = \bs{\mu}(\Delta_e)$ and $\ell_e(\bs{\mu}(\b{x})(v)) = \bs{\mu}(\b{x}_e)(v)$ for $v\neq e$.
\end{lemma}

\begin{proof} It suffices to show for $\bs{\mu}=\mu_u$ with $u\neq e$.
It is rather clear from the definition that $\mu_u$ commutes with the deletion at $e$.
The map $\ell_e$ commutes with mutations because $\ell_e$ is a homomorphism.
\end{proof}

\begin{corollary} \label{C:liftcluster} For any extended cluster $\b{x}_e^\dag$ of $(\Delta_e, \b{x}_e)$, there is an extended cluster $\b{x}^\dag$ of $(\Delta, \b{x})$
such that $\b{x}^\dag(e)=e$ and $\ell_e (\b{x}^\dag(v))=\b{x}_e^\dag(v)$ for $v\neq e$.
In particular, $\mc{C}(\Delta_e, \b{x}_e; \bs{\sigma}_e)\subseteq (\mc{C}(\Delta, \b{x};\bs{\sigma})_e)_{\ep^\perp}$.
\end{corollary}

\begin{proof} Suppose that $\b{x}_e^\dag$ is obtained from $\b{x}_e$ by a sequence of mutations $\bs{\mu}:=\mu_{u_k}\cdots\mu_{u_2}\mu_{u_1}$.
We claim that $\b{x}^\dag=\bs{\mu}(\b{x})$.
Since the mutation is away from $e$, we have that $\b{x}^\dag(e)=e$.
By Lemma \ref{L:projmu} we have that
$\ell_e(\b{x}^\dag(v))= \ell_e(\bs{\mu}(\b{x})(v))=\bs{\mu}(\b{x}_e)(v).$
\end{proof}

\begin{proposition} \label{P:subalgebra} Suppose that $e$ is frozen or $B(\Delta)$ has full rank.
Let $\mc{A}:=\uca(\Delta,\b{x};\bs{\sigma})$ and $\mc{S}:=\uca(\Delta_e, \b{x}_e ;\bs{\sigma}_e)$,
then we have that $(\mc{A}_e)_{\ep^\perp}\cong \mc{S}$.
\end{proposition}

\begin{proof} By Lemma \ref{L:reduce2frozen}, we only need to consider the case when $e$ is frozen.
We first show that $\mc{S}\subseteq (\mc{A}_e)_{\ep^\perp}$.
Let $\b{x}^\dag\sim \b{x}$ be any extended cluster of $(\Delta,\b{x})$.
Since $e$ is frozen, $\ell_e(\b{x}^\dag)\setminus\{1\}$ is an extended cluster of $(\Delta_e, \b{x}_e)$ by Lemma \ref{L:projmu}.
For any homogeneous $s\in \mc{S}$, we write $s$ as a Laurent polynomial $p(\ell_e(\b{x}^\dag))$, which is equal to $\ell_e (p(\b{x}^\dag))$.
Let $\b{x}^\ddag$ be another extended cluster, and we write $s$ as a Laurent polynomial $p'(\ell_e(\b{x}^\ddag))=\ell_e (p'(\b{x}^\ddag))$.
By the definition of $\ell_e$, we see that $p(\b{x}^\dag)$ and $p'(\b{x}^\ddag)$ differ by a factor of some power of $e$.
So $s$ lies in the image $\ell_e(\mc{A}_e)$, which is $(\mc{A}_e)_{\ep^\perp}$.

Now let $a\in (\mc{A}_e)_{\ep^\perp}$ be homogeneous.
Let $\b{x}_e^\dag \sim \b{x}_e$ be any extended cluster of $(\Delta_e,\b{x}_e)$.
Choose an extended cluster $\b{x}^\dag$ as in Corollary \ref{C:liftcluster}, and then we can write $a$ as a Laurent polynomial $p(\b{x}^\dag)$.
We apply the map ${\ell}_e$ to $a=p(\b{x}^\dag)$, and then $a=\ell_e(a)=p(\ell_e(\b{x}^\dag))$ is a Laurent polynomial in $\b{x}_e^\dag$.
We thus obtain the other containment $(\mc{A}_e)_{\ep^\perp}\subseteq\mc{S}$.
\end{proof}

\begin{remark} \label{R:cluster} We saw from Corollary \ref{C:liftcluster} that $\mc{C}(\Delta_e, \b{x}_e;\bs{\sigma})\subseteq (\mc{C}(\Delta, \b{x};\bs{\sigma})_e)_{\ep^\perp}$.
However, we have examples to show that under the assumption of Proposition \ref{P:subalgebra} the inclusion can be strict.
\end{remark}


We define a linear function $\ep_l: \mb{R}^{\Delta_0} \to \mb{R}$ by
$$\ep_l: \g \mapsto \innerprod{\ep, \g\bs{\sigma} }.$$
We define a linear isomorphism $\i_e: \mb{R}^{(\Delta_e)_0} \to \Ker \ep_l$ by
$$\e_v \mapsto \e_v - \innerprod{\ep,\bs{\sigma}(v)}\e_e.$$
By construction the map $\i_e$ is weight-preserving, that is,
$\g\bs{\sigma}_e = \i_e(\g)\bs{\sigma}$, and satisfies that
$i_e (\b{x}_e^{\g}) = \b{x}^{\i_e(\g)}$ ($i_e$ as defined before Definition \ref{D:projector}).
Let $I_e$ be the matrix of the linear transform $\i_e$.
\begin{lemma} \label{L:projB} If $e$ is frozen, then $B(\Delta_e) I_e = B(\Delta)$.
So we have that ${i}_e(y_u') = y_u$ and $\ell_e(y_u)=y_u'$, where $y_u$ and $y_u'$ are the $y$-variables associated to $(\Delta,\b{x})$ and $(\Delta_e,\b{x}_e)$.
\end{lemma}

\begin{proof} It suffices to show that the $e$-th columns are equal, or equivalently
$$b_{u,e} = \sum_{v\neq e} b_{u,v} I_{v,e}.$$
Since $\sum_v b_{u,v}\bs{\sigma}(v)=0$,
we have that $\sum_{v\neq e} b_{u,v}\bs{\sigma}(v)=b_{u,e}\bs{\sigma}(e)$.
The required identity is obtained by taking $\innerprod{\bs{\sigma}(e),-}$ on the both sides.

By definition, ${i}_e(y_u') = {i}_e(\b{x}_e^{-b_u'})= \b{x}^{-\i_e(b_u')} = \b{x}^{-b_u} = y_u$, where $b_u'$ is the $u$-th row of $B(\Delta_e)$. Applying ${\ell}_e$ to the above equation, we get that $\ell_e(y_u)=y_u'$.
\end{proof}

\noindent Let $\ll_e: \mb{R}^{\Delta_0} \to \mb{R}^{(\Delta_e)_0} $ be the linear map defined by forgetting the $e$-th coordinate.
Then it is a left inverse of $\i_e$ and satisfies that
$\ell_e (\b{x}^{\g}) = \b{x}_e^{\ll_e(\g)}$.
We say that an upper cluster algebra $\uca(\Delta)$ is {\em $\g$-indexed} if
it has a basis whose elements have well-defined distinct $\g$-vectors in $G(\Delta)$ (see Section \ref{s:gv}).

\begin{corollary} \label{C:Gproj_e} Suppose that $e$ is frozen or $B(\Delta)$ has full rank.
If $\uca(\Delta,\b{x};\bs{\sigma})$ is $\g$-indexed, then $\uca(\Delta_e,\b{x}_e;\bs{\sigma}_e)$ has a basis with well-defined $\g$-vectors in $\ll_e(G(\Delta))$. 
Moreover, if $e$ is a projector, then $\uca(\Delta_e,\b{x}_e;\bs{\sigma}_e)$ is $\g$-indexed as well.
\end{corollary}

\begin{proof} Suppose that a basis element $z$ in $\uca(\Delta;\bs{\sigma})$ can be written as \eqref{eq:z}.
Then by Lemma \ref{L:projB}, $\ell_e(z)= \b{x}_e^{\ll_e(\g(z))} F({y}_1',{y}_2',\dots,{y}_p')$.
By Proposition \ref{P:subalgebra}, $\uca(\Delta_e;\bs{\sigma}_e)$ has a basis with well-defined $\g$-vectors in $\ll_e(G(\Delta))$.
Now suppose that $e$ is a projector, then $\uca(\Delta;\bs{\sigma})_{\ep^\perp}=(\uca(\Delta;\bs{\sigma})_e)_{\ep^\perp}=\uca(\Delta_e;\bs{\sigma}_e)$.
We remain to show that $\g$-vectors can be chosen to be distinct for a basis of $\uca(\Delta_e;\bs{\sigma}_e)$.
Since elements with well-defined $\g$-vectors are homogeneous,
$\uca(\Delta;\bs{\sigma})_{\ep^\perp}$ has a basis with $\g$-vectors in $G(\Delta) \cap \Ker \ep_l$.
It suffices to show that $\ll_e$ is injective restricted to $G(\Delta) \cap \Ker \ep_l$.
If $\g\in \Ker \ll_e$, then $\g(u)=0$ for all $u\neq e$. But $\e_e\notin \Ker \ep_l$.
The last statement follows.
\end{proof}

\subsection{The Cases when $\ep$ is not necessary in $\mb{M}$} \label{s:notin}

In general, the cases when $\ep$ is not necessary in $\mb{M}$ is much more complicated.
We can only obtain similar result after imposing more restriction.
However, in what follows we need not assume $\ep$ to be real, so we use $\theta$ instead of $\ep$.

\begin{definition} \label{D:attached}
A set of vertices $\b{u}$ is called {\em attached} to another set of vertices $\b{v}$ if $t(a) \in \b{v}$ (resp. $h(a)\in \b{v}$) whenever $h(a) \in \b{u}$ (resp. $t(a)\in \b{u}$).
For such a pair $(\b{u},\b{v})$, we form a quiver $\Deltauv$ by deleting $\b{v}$ and freezing $\b{u}$.
\end{definition}

Let $\b{v}$ be a set of vertices. We write $\b{x}(\hat{\b{v}})$ for the set of cluster variables $\{\b{x}(v)\mid v\notin\b{v} \}$ in the cluster $\b{x}$.
We also write $\bs{\sigma}(\hat{\b{v}})$ for the restriction of $\bs{\sigma}$ to $\Delta_0\setminus \b{v}$.

\begin{proposition} \label{P:outerproj} Let $\mc{A}:=\uca(\Delta,\b{x};\bs{\sigma})$ with $B(\Delta)$ of full rank.
Suppose that $\theta$ is a weight such that $\innerprod{\theta,\bs{\sigma}(v)}= 0$ for all $v\in \Delta_0$ except for a set of frozen vertices $\b{v}$,
where all $\innerprod{\theta,\bs{\sigma}(v)}<0$ $($or all $>0)$.
Then the graded subalgebra $\mc{A}_{\theta^\perp}$ is equal to $\mc{S}:=\uca(\Deltauv,\b{x}(\hat{\b{v}});\bs{\sigma}(\hat{\b{v}}))$,
where $\b{u}$ is attached to $\b{v}$.
\end{proposition}

\begin{proof} The same argument as Lemma \ref{L:reduce2frozen} shows that $\mc{S}$ is contained in $\mc{A}$, and thus in $\mc{A}_{\theta^\perp}$.
Indeed, by Theorem \ref{T:bounds} it suffices to show that any $s\in\mc{S}$ can be written as a Laurent polynomial in $\b{x}':=\mu_u(\b{x})$ for each $u\in \b{u}$.
Note that $s$ can be written as a Laurent polynomial in $\b{x}$, which is polynomial in $\b{x}(u)$'s.
Let $h$ be the exchange polynomial $\prod_{u\to v}\b{x}(v) + \prod_{w\to u}\b{x}(w)$ so that $\b{x}(u)=h\b{x}'(u)^{-1}$.
We thus get the required Laurent polynomial in $\b{x}'$.

Conversely, for any homogeneous $a\in \mc{A}_{\theta^\perp}$, we can write $a$ as a Laurent polynomial $p$ in any given extended cluster $\b{x}^\dag \sim \b{x}$ with polynomial coefficient variables.
We assume that $(\Delta^\dag,\b{x}^\dag;\bs{\sigma}^\dag)$ is obtained from $(\Delta,\b{x};\bs{\sigma})$ by mutations away from $\b{u}$.
It is clear that any cluster mutated from the seed $(\Deltauv,\b{x}(\hat{\b{v}}))$ is contained in some $\b{x}^\dag$.
Since $\b{u}$ is attached to $\b{v}$, we still have that $\innerprod{\theta,\bs{\sigma}^\dagger(v)}= 0$ for all $v\in \Delta_0\setminus \b{v}$
and $\innerprod{\theta,\bs{\sigma}(v)}<0$ (or $>0$) for $v\in \b{v}$.
Then we see that any coefficient variables in $\b{x}(\b{v})$ cannot appear in $p$ because otherwise $a \notin \mc{A}_{\theta^\perp}$.

We remain to show that $p$ is polynomial in $\b{x}(u)$ for $u\in \b{u}$.
We write $a$ as a Laurent polynomial $p'$ in the mutated cluster $\b{x}':=\mu_u(\b{x})$.
Since $\b{u}$ is attached to $\b{v}$, the exchange polynomial $h$ contains variables in $\b{x}(\b{v})$.
In particular, $\innerprod{\theta,\bs{\sigma}'(u)}$ has the same sign as $\innerprod{\theta,\bs{\sigma}(v)}$ where $v\in\b{v}$.
We decompose $p'$ as a sum of (homogeneous) Laurent monomials with no cancelations.
We see that $\b{x}'(u)$ must have non-positive degree in each such monomial because otherwise $a \notin \mc{A}_{\theta^\perp}$.
But $\b{x}'(u)^{-1} = \b{x}(u)h^{-1}$ so $p$ must be polynomial in $\b{x}(u)$.
\end{proof}

\begin{remark} Similar to Remark \ref{R:CR1}, we can replace the full rank condition by the condition that $\mc{A}$ is Noetherian and $(\Delta,\b{x})$ is a CR1 seed in $\mc{A}$.
Then by \eqref{eq:CR1} $(\Deltauv,\b{x}(\hat{\b{v}}))$ is also a CR1 seed in $\mc{A}$, so by Lemma \ref{L:RCA} we still have that $\mc{S} \subset \mc{A}$.
The same remark applies to several results depending on Proposition \ref{P:outerproj}, such as
Corollary \ref{C:Gprojv}, Theorem \ref{T:removal}, Theorem \ref{T:hat}, and Corollary \ref{C:removal_model}.
\end{remark}

\begin{remark} Analogous to Remark \ref{R:cluster}, under the assumption of Proposition \ref{P:outerproj}, we only have
$\mc{C}(\Deltauv, \b{x}(\hat{\b{v}});\bs{\sigma}(\hat{\b{v}}))\subseteq \mc{C}(\Delta, \b{x};\bs{\sigma})_{\theta^\perp}$ in general.
\end{remark}

In the situation of Proposition \ref{P:outerproj},
if we define ${p}_{\b{v}}: \mc{A}\to \mc{A}_{\theta^\perp}=\mc{S}$ by evaluation of $\b{x}(\b{v})$ at $0$,
then $p_{\b{v}}$ is a left inverse of the embedding $i_{\b{v}}:\mc{S} \hookrightarrow \mc{A}$.
Similar to Lemma \ref{L:projB}, we have that $i_{\b{v}}(y_u')=y_u$ and $p_{\b{v}}(y_u)=y_u'$, where $y_u'$ is the $y$-variable for $(\Deltauv,\b{x}(\hat{\b{v}}))$.
Let $\pp_{\b{v}}:\mb{R}^{\Delta_0} \to \mb{R}^{(\Delta_{\b{v}})_0}$ and $\i_{\b{v}}: \mb{R}^{(\Delta_{\b{v}})_0}\to \mb{R}^{\Delta_0}$ be the linear maps such that
${p}_{\b{v}}(\b{x}^{\g})=\b{x}(\hat{\b{v}})^{\pp_{\b{v}}(\g)}$ and ${i}_{\b{v}}(\b{x}(\hat{\b{v}})^{\g})=\b{x}^{\i_{\b{v}}(\g)}$.
The proof of the following corollary is similar to (but easier than) that of Corollary \ref{C:Gproj_e}.

\begin{corollary} \label{C:Gprojv} Suppose that we are in the situation of Proposition \ref{P:outerproj}.
If $\uca(\Delta;\bs{\sigma})$ is $\g$-indexed,
then $\uca(\Deltauv;\bs{\sigma}(\hat{\b{v}}))$ is $\g$-indexed as well, and moreover we have that $G(\Deltauv)= \pp_{\b{v}}(G(\Delta))$.
\end{corollary}

\section{Projections from Semi-invariant Rings} \label{S:SI}

\subsection{Semi-invariant Rings of Quiver Representations}

Let us briefly recall the semi-invariant rings of quiver representations \cite{S1}.
Let $Q$ be a finite quiver without oriented cycles. We fix an algebraically closed field $k$ of characteristic zero.
For a dimension vector $\beta$ of $Q$, let $V$ be a $\beta$-dimensional vector space $\prod_{i\in Q_0} k^{\beta(i)}$. We write $V_i$ for the $i$-th component of $V$.
The space of all $\beta$-dimensional representations is
$$\Rep_\beta(Q):=\bigoplus_{a\in Q_1}\Hom(V_{t(a)},V_{h(a)}).$$
The product of general linear group
$$\GL_\beta:=\prod_{i\in Q_0}\GL(V_i)$$
acts on $\Rep_\beta(Q)$ by the natural base change.
Define $\SL_\beta\subset \GL_\beta$ by
$$\SL_\beta=\prod_{i\in Q_0}\SL(V_i).$$
We are interested in the rings of semi-invariants
$$\SI_\beta(Q):=k[\Rep_\beta(Q)]^{\SL_\beta}.$$
The ring $\SI_\beta(Q)$ has a weight space decomposition
$$\SI_\beta(Q)=\bigoplus_\sigma \SI_\beta(Q)_\sigma,$$
where $\sigma$ runs through the multiplicative {\em characters} of $\GL_\beta$.
We refer to such a decomposition the $\sigma$-grading of $\SI_\beta(Q)$.
Recall that any character $\sigma: \GL_\beta\to k^*$ can be identified with a weight vector
$\sigma \in \mb{Z}^{Q_0}$
\begin{equation*} \label{eq:char} \big(g(i)\big)_{i\in Q_0}\mapsto\prod_{i\in Q_0} \big(\det g(i)\big)^{\sigma(i)}.
\end{equation*}

Let us understand these multihomogeneous components
$$\SI_\beta(Q)_\sigma:=\{f\in k[\Rep_\beta(Q)]\mid g(f)=\sigma(g)f, \forall g\in\GL_\beta \}.$$
Since $Q$ has no oriented cycles, the degree zero component is the field $k$ \cite{Ki}.
For any projective presentation $f: P_1\to P_0$, we view it as an element in the homotopy category $K^b(\proj Q)$ of bounded complexes of projective representations of $Q$.
We view each weight $\sigma \in \mb{Z}^{Q_0}$ as an element in the Grothendieck group $K_0(K^b(\proj Q))$.
Concretely, suppose that $P_1=P(\sigma_1)$ and $P_0=P(\sigma_0)$ for $\sigma_1,\sigma_0\in\mathbb{Z}_{\geq 0}^{Q_0}$, then $\sigma = \sigma_1-\sigma_0$.
Here, we use the notation $P(\sigma)$ for $\bigoplus_{i\in Q_0} \sigma(i) P_i^{}$, where $P_i$ is the indecomposable projective representation corresponding to the vertex $i$.
Let $\sigma$ be a weight such that $\sigma\cdot \beta=0$, and $f$ be a projective presentation of weight $\sigma$.
In \cite{S1} Schofield constructed for each such $f$ a (not necessary non-zero) element in $\SI_\beta(Q)_\sigma$.
For the construction of $s(f)$, we refer the readers to the original text or \cite[Section 1]{Fs1}.
In fact,
\begin{theorem}[\cite{DW1,SV,DZ}] \label{T:inv_span} $s(f)$'s span $\SI_\beta(Q)_{\sigma}$ over the base field $k$.
\end{theorem}

Let $\Sigma_\beta(Q)$ be the set of all weights $\sigma$ such that $\SI_\beta(Q)_\sigma$ is non-empty.
It is known that such a weight $\sigma$ corresponds to some dimension vector $^\alpha\sigma$, i.e., $\sigma\cdot(-)=-\innerprod{{^\alpha\sigma},-}_Q$,
where $\innerprod{-,-}_Q$ is the {\em Euler form} of $Q$.
The Euler form induces a form $\innerprod{-,-}$ on the space of weights satisfying
$\innerprod{\sigma,\theta}=\innerprod{{^\alpha\sigma},{^\alpha\theta}}_Q$.
The unimodular lattice $\mb{Z}^{Q_0}$ with this form will be used in our weight configuration later.
The next theorem is an easy consequence of King's stability criterion \cite[Proposition 3.1]{Ki} and Theorem \ref{T:inv_span}.
By a general $\beta$-dimensional representation, we mean in a sufficiently small Zariski open subset (``sufficient" here depends on the context).
Following \cite{S2}, we use the notation $\gamma\hookrightarrow\beta$ to mean that
a general $\beta$-dimensional representation has a $\gamma$-dimensional subrepresentation.

\begin{theorem} \cite[Theorem 3]{DW1} \label{T:cone} We have
$$\Sigma_\beta(Q)=\{\sigma\in \mb{Z}^{Q_0} \mid \sigma\cdot \beta=0 \text{ and } \sigma\cdot \gamma\geq 0 \text{ for all } \gamma\hookrightarrow \beta\}.$$
In particular, $\Sigma_\beta(Q)$ is a pointed saturated monoid.
\end{theorem}

Recall that a dimension vector $\alpha$ is called a {\em real} root if $\innerprod{\alpha,\alpha}_Q=1$.
It is called {\em Schur} if $\Hom_Q(M,M)=k$ for
$M$ general in $\Rep_\alpha(Q)$. For a real Schur root $\alpha$, $\Rep_{n\alpha}(Q)$ has a dense orbit for any $n\in \mb{N}$.

\begin{lemma}[{\cite[Lemma 1.7]{Fs1}}] \label{L:gsreal}
If $\alpha$ is a real Schur root, then $\SI_\beta(Q)_{-\innerprod{n\alpha,-}_Q}\cong k$ for any $n\in \mb{N}$.
\end{lemma}

\begin{lemma}\cite[Lemma 1]{DW1} \label{L:exact} Suppose that we have an exact sequence of representations of $Q$
$$0\to L \to M \to N\to 0.$$
If $\innerprod{\dimbar L,\beta}=\innerprod{\dimbar N,\beta}=0$, then as a function on $\Rep_\beta(Q)$, $s(f_M)$ is, up to a scalar, equal to $s(f_L)s(f_N)$, where $f_M$ is the minimal presentation of $M$.
If $\innerprod{\dimbar L,\beta}> 0$ or $\innerprod{\dimbar N,\beta}< 0$,
then $s(f_M)=0$.
\end{lemma}

\subsection{Orthogonal Projections}

For two dimension vector $\alpha,\beta$, we write $\alpha \perp \beta$ if $\Hom_Q(M,N)=\Ext_Q(M,N)=0$ for $M,N$ general in $\Rep_\alpha(Q)$ and $\Rep_\beta(Q)$.
We recall that an {\em exceptional sequence of dimension vector} $\mb{E}:=\{\ep_1,\ep_2,\dots, \ep_n\}$ is a sequence of real Schur roots of $Q$ such that $\ep_i\perp \ep_j$ for any $i<j$.
It is called {\em quiver} if $\innerprod{\ep_j,\ep_i}_Q\leq 0$ for any $i<j$.
It is called {\em complete} if $n=|Q_0|$.
The {\em quiver} $Q(\mb{E})$ of a quiver exceptional sequence $\mb{E}$ is by definition the quiver with vertices labeled by $\ep_i$ and $\ext_Q(\ep_j,\ep_i)$ arrows from $\ep_j$ to $\ep_i$.
According to \cite[Theorem 4.1]{S2}, $\ext_Q(\ep_j,\ep_i)=-\innerprod{\ep_j,\ep_i}_Q$.

Given an exceptional sequence of dimension vector $\mb{F}=\{\vep_1,\vep_2,\dots,\vep_r \}$ with $F_i$ general in $\Rep_{\vep_i}(Q)$,
we consider the (right) {\em orthogonal subcategory} $\mb{F}^\perp$ defined by
$$\mb{F}^\perp:=\{M\in\Rep(Q) \mid F_i\perp M \text{ for } i=1,2,\dots,r \}.$$
According to \cite{S1}, this abelian subcategory is equivalent to the module category of another quiver denoted by $Q_{\mb{F}}$.
We say that $Q$ {\em projects} to $Q_{\mb{F}}$ through $\mb{F}$.
In this case there is a unique quiver exceptional sequence $\mb{E}=\{\ep_1,\ep_2,\dots, \ep_{|Q_0|-r}\}$ of $Q$ such that
the concatenation of two sequences $(\mb{F},\mb{E})$ is a (complete) exceptional sequence and $Q(\mb{E})\cong Q_{\mb{F}}$.
We refer to \cite{S1} for the construction of $\mb{E}$.

We define a linear isometry (with respect to the Euler forms) $\iota_{\mb{F}}: K_0(Q_{\mb{F}})\to K_0(Q)$ by
\begin{equation} \label{eq:dimemb} \beta' \mapsto \sum \beta'(i)\ep_i.\end{equation}
Conversely we define a linear map $\pi_{\mb{F}}$ left inverse to $\iota_{\mb{F}}$ as the composition $K_0(Q)\to K_0(\mb{F}^\perp) \cong K_0(Q_{\mb{F}})$,
where the first map $\beta_0:=\beta \mapsto \beta_r$ is given by the recursion
$$\beta_{i+1} = \beta_i- \innerprod{\vep_{i+1},\beta_i}_Q \vep_{i+1}.$$
Both maps induce a linear map on the weights, still denoted by $\iota_{\mb{F}}$ and $\pi_{\mb{F}}$.
We write $\beta_{\mb{F}}$ and $\sigma_\mb{F}$ for $\pi_{\mb{F}}(\beta)$ and $\pi_{\mb{F}}(\sigma)$.
To effectively compute $\sigma_\mb{F}$ later, we can use the formulae
\begin{equation} \label{eq:wtproj} \sigma_\mb{F}(i) = \sigma\cdot \ep_i.
\end{equation}


\begin{theorem}[{\cite[Theorem 2.39]{DW2}}] \label{T:SIemb} Let $\ep$ be a real Schur root such that $\ep\perp\beta$.
Then we have that $\SI_{\beta_\ep}(Q_{\ep})\cong \SI_{\beta}(Q)_{\ep^\perp}$.
\end{theorem}

\begin{proposition} \label{P:extremal_proj} Let $\ep$ is a real root of $Q$. Then $^\sigma \ep$ is extremal in $\cone$ if and only if $^\sigma\ep$ is a projector for $\SI_\beta(Q)$.
\end{proposition}

\begin{proof} By \cite[Theorem 3.17]{PV}, any $\SI_\beta(Q)$ is a UFD.
Since $\SI_\beta(Q)_{\b{0}}=k$ and $\Sigma_\beta(Q)$ is pointed, the direction $``\Leftarrow"$ follows from Proposition \ref{P:projector}.(3). Conversely, let us first show that $\ep$ is in fact Schur.
If not, then a general representation $E$ in $\Rep_\ep(Q)$ decomposes as $E=L\oplus N$.
By Lemma \ref{L:exact}, we must have that $s(f_E)=s(f_L)s(f_N)\neq 0$.
Since $^\sigma \ep$ is extremal, the weights of $f_L$ and $f_N$ must lie on $\mb{R}^+ {^\sigma}\ep$.
But $\ep$ is real, so it is clearly impossible.

For any dimension vector $\alpha$ such that $\alpha\perp\beta$ and $\innerprod{\ep,\alpha}_Q>0$, we must have that $h\ep\hookrightarrow \alpha$ where $h=\hom_Q(\ep,\alpha)$.
Indeed, let $\gamma$ be the rank of a general morphism $f$ from a general representation $U\in\Rep_{h\ep}(Q)$ to a general representation $V\in\Rep_{\alpha}(Q)$.
Note that $h\geq \innerprod{\ep,\alpha}_Q>0$.
We claim that $\gamma = h\ep$.
We first observe that $\gamma$ must be a multiple of $\ep$, say $\gamma = h'\ep$.
Otherwise by Lemma \ref{L:exact} the weights of $\Ker(f)$ and $\Img(f)$ lie in $\cone$ contradicting $^\sigma\ep$ being extremal.
Now we have an exact sequence
$0\to h'E\xrightarrow{f\mid_{h'E}} V \to W \to 0$ where $E$ is the exceptional representation of dimension $\ep$.
By applying $\Hom_Q(E,-)$, we find that $\hom_Q(E,W)=h-h'$.
Since $\ep$ is real Schur, a general representation of dimension $h\ep$ is isomorphic to $hE=h'E \oplus (h-h')E$.
If $h>h'$, then there is a nontrivial homomorphism from $(h-h')E$ to $W$.
We can thus construct a homomorphism from $hE$ to $V$ with image strictly containing $f(h'E)$.
So we must have that $h=h'$.

Now by Theorem \ref{T:inv_span}, any semi-invariant function of weight $^\sigma \alpha$ is a linear combination of $s(f)$'s,
where $f$ is the minimal presentation of a general representation in $\Rep_\alpha(Q)$.
By Lemma \ref{L:exact}, any such $s(f)$ has a factor $s(f_E)^h$ in $\SI_\beta(Q)$.
We conclude that $^\sigma\ep$ is a projector from Lemma \ref{L:gsreal} and \ref{L:projector}.
\end{proof}

We say that a semi-invariant ring $\SI_\beta(Q)$ is a {\em naturally graded} upper cluster algebra $\uca(\Delta,\b{x}; \bs{\sigma})$ if $\SI_\beta(Q)\cong \uca(\Delta,\b{x}; \bs{\sigma})$ and $\bs{\sigma}(v)$ is the $\sigma$-weight of $\b{x}(v)$ under the isomorphism.
\begin{theorem} \label{T:invGUCA} Suppose that $\SI_\beta(Q)$ is a naturally graded UCA $\uca(\Delta,\b{x}; \bs{\sigma})$, and $B(\Delta)$ has full rank.
If $\bs{\sigma}(e)={^\sigma}\ep$ is real and extremal in $\cone$,
then $\ep$ is Schur and $\SI_{\beta_\ep}(Q_\ep)$ is the naturally graded UCA $\uca(\Delta_e,\b{x}_e; \pi_\ep(\bs{\sigma}))$.
\end{theorem}

\begin{proof} From the proof of Proposition \ref{P:extremal_proj}, we have seen that $\ep$ must be Schur.
The rest follows from Proposition \ref{P:subalgebra}, \ref{P:extremal_proj}, and Theorem \ref{T:SIemb}.
\end{proof}

\begin{remark} In what follows, we will always identify $\bs{\sigma}_e$ with $\pi_\ep(\bs{\sigma})$ through the linear isometry $\iota_\ep$ if available.
So we may keep $\uca(\Delta_e,\b{x}_e; \bs{\sigma}_e)$ instead of writing $\uca(\Delta_e,\b{x}_e; \pi_\ep(\bs{\sigma}))$.
\end{remark}

\begin{remark} We note that if $\vep_1,\vep_2,\dots,\vep_r$ is an exceptional sequence with each $^\sigma \vep_i$ extremal in $\cone$,
then each $^\sigma(\vep_2)_{\vep_1},\dots,{^\sigma}(\vep_{n})_{\vep_1}$ is extremal in $\mb{R}^+\Sigma_{\beta_{\vep_1}}(Q_{\vep_1})$.
So Theorem \ref{T:invGUCA} can be inductively applied to such an exceptional sequence.
\end{remark}

\subsection{Complete Flag Varieties, Unipotent Subgroups, and Sextuple Flags}

We first recall the main result in \cite{Fs1}.
Consider the triple flag quiver $T_n$
$$\tripleflag$$
with the {\em standard} dimension vector $\beta_n$ as indicated above.
Let $\e_i^a$ be the unit vector supported on the $i$-th vertex of the $a$-th arm of $T_n$.
Let $\f_{i,j}$ be the weight given by $\e_n - \e_i^1 - \e_j^2 - \e_k^3$ for $k=n-i-j$,
and $\vep_{i,j}={^\alpha\f_{i,j}}$. We use the convention that $\e_0^a$ is the zero vector and $\e_n^a=\e_n$ for $a=1,2,3$.
It is clear that each $\f_{i,j}$ is extremal in $\mb{R}^+\Sigma_{\beta_n}(T_n)$.

It is proved in \cite{Fs1} that the semi-invariant ring $\SI_{\beta_n}(T_n)$ is the GUCA
$\uca(\Delta_n;\bs{\sigma}_n)$, where $\Delta_n$ is the ice hive quiver and $\bs{\sigma}_n(i,j)=\f_{i,j}$.
The {\em ice hive quiver $\Delta_n$ of size $n$} is an ice quiver in the plane with $\sm{n+2\\ 2}-3$ vertices arranged in a triangular grid.
We label the vertices as shown in Figure~\ref{fig:hive}.
Note that three vertices $(0,0),(n,0)$ and $(0,n)$ are missing, and all boundary vertices are frozen.
There are three types of arrows: $(i,j)\to (i+1,j), (i,j)\to (i,j-1)$, and $(i,j) \to (i-1,j+1)$.
It is easy to see that $B(\Delta_n)$ has full rank.
\begin{figure}[h]
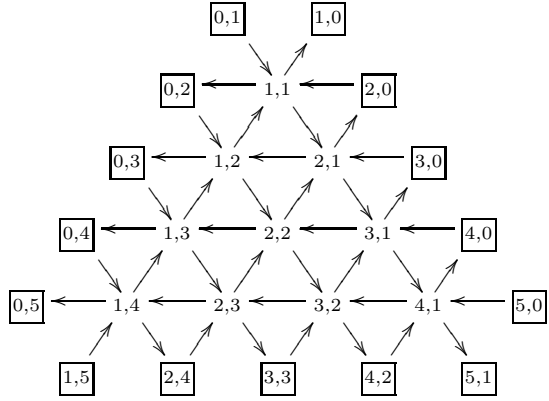

\begin{center}
$\hivesix$
\end{center}
\caption{The ice hive quiver $\Delta_6$.}
\label{fig:hive}
\end{figure}

\begin{example}[Complete flag varieties] \label{ex:cflag}
Consider the exceptional sequence $\mb{F}:=\{\vep_{i,n-i} \}_{1\leq i\leq n-1}$ of $T_n$.
Note that $\vep_{i,n-i} \perp \vep_{j,n-j}$ for $i\neq j$ so the order of $\mb{F}$ does not matter.
We find that the unique quiver exceptional sequence $\mb{E}$ completing $\mb{F}$ is given by
$$\{ \ep_n, \dots, \ep_2, \ep_{1},\e_{n-1}^3,\dots, \e_2^3, \e_{1}^3 \},$$
where $\ep_{i} = {^\alpha} (\e_{n}-\e_i^1-\e_{n-i+1}^2)$.
The quiver of $\mb{E}$ is displayed below, where $\e_1^3, \dots, \e_{n-1}^3$ and $\ep_1,\dots, \ep_{n}$ corresponds to the vertices $1,\dots,n-1$, and $1',\dots,n'$ respectively.
$$\flagstar$$

We denote the quiver of $\mb{E}$ by $SF_n$. Let $\beta_n'$ be the dimension vector given by $\beta_n'(i)=i$ and $\beta_n'(i')=1$.
We can verify by \eqref{eq:dimemb} that under the embedding the dimension vector $\beta_n'$ goes to $\beta_n$.
It is not hard to see that $\SI_{\beta_n'}(SF_n)$ is the {\em total coordinate ring} $\mc{R}(\Fl_n)$ of the complete flag variety $\Fl_n$ of the $n$-dimensional vector space. Let $\op{Cl}(X)$ be the divisor class group of a variety $X$. By definition,
$$\mc{R}(X):=\bigoplus_{[D]\in\op{Cl}(X)}\Gamma(X,O_X(D)),$$
where the multiplication is given by multiplying homogeneous sections in the field of rational functions $k(X)$.
However, to rigorously prove this fact, we need {\em variation of GIT}. We will write this up in some lecture notes.

We observe that the exceptional sequence $\mb{F}$ corresponds to the last row of frozen vertices of $\Delta_n$.
Let $\Delta_n^\flat$ be the ice quiver obtained from $\Delta_n$ by deleting the last row of frozen vertices (see Figure \ref{fig:flats}).
\begin{figure}[h]
        \centering
        \begin{subfigure}[h]{0.3\textwidth}
            $\hivesixf$  \label{fig:hivef}
        \end{subfigure}                 {\quad \hfill}
        \begin{subfigure}[h]{0.3\textwidth}
            $\hivesixff$  \label{fig:hiveff}
        \end{subfigure}                {\qquad\qquad \hfill} \\
        \caption{The ice quiver $\Delta_6^\flat$ (left) and $\Delta_6^{\flat\flat}$ (right)}\label{fig:flats}
\end{figure}
It follows from Theorem \ref{T:invGUCA} that
\begin{proposition}The total coordinate ring of $\Fl_n$ is the GUCA $\uca(\Delta_n^\flat;\bs{\sigma}_n^\flat)$.
\end{proposition}
\noindent Here, the new weight configuration $\bs{\sigma}_n^\flat$ can be computed from \eqref{eq:wtproj}.
For example, $\bs{\sigma}_n^\flat(i,0)= \sum_{k=i+1}^n\e_k' - \e_{n-i}$, which is clearly extremal in $\mb{R}^+\Sigma_{\beta_n'}(SF_n)$.
\end{example}

\begin{example}[The unipotent subgroups] \label{ex:ucell}
We sketch how to get the cluster structure of the maximal unipotent subgroup $\op{N}_n$ of $\GL_n$ from the previous example, and leave the details to the readers.
Although as an ungraded algebra it is just a polynomial ring, it is certainly non-trivial if the multigrading is taken into account.

We verify that each $\bs{\sigma}_n^\flat(i,0)$ corresponds to a real Schur root, denoted by $\vep_i$.
In fact, $\mb{F}:=\{\vep_{n-1},\dots,\vep_{2},\vep_{1} \}$ is an exceptional sequence of $SF_n$.
We find that the unique quiver exceptional sequence $\mb{E}$ completing $\mb{F}$ is given by
$\{ \ep_n', \dots, \ep_2',\ep_{1}' \},$
where $\ep_{i}' = \e_{n-1+i}' + \sum_{i}^{n-1}\e_i$.
Then we can verify that the arrow matrix of $Q(\mb{E})$ is a strict upper-triangular matrix with all entries equal to one.

Let $\b{1}$ be the dimension vector $(1,1,\dots,1)$.
We can verify by \eqref{eq:dimemb} that under the embedding the dimension vector $\b{1}$ goes to $\beta_n'$.
Clearly $\SI_{\b{1}}(Q(\mb{E}))$ is the coordinate ring of $\op{N}_n$.
Let $\Delta_n^{\flat\flat}$ be the ice quiver obtained from $\Delta_n^\flat$ by deleting the right (or left) row of frozen vertices.
It follows that
\begin{proposition} The coordinate ring of $\op{N}_n$ is the GUCA $\uca(\Delta_n^{\flat\flat};\bs{\sigma}_n^{\flat\flat})$.
\end{proposition}

\noindent We want to tell curious readers that we cannot project from any vertex of $\Delta_n^{\flat\flat}$ except for $(n,1)$ because none of the weights in $\bs{\sigma}_n^{\flat\flat}$ corresponds to a real Schur root.
\end{example}

\begin{example}[Complete sextuple flags] \label{ex:6flag}
Consider the exceptional sequence $\mb{F}:=\{\vep_{0,n}, \vep_{n,0}, \vep_{n,n} \}$ of $T_{2n}$.
We find that the unique quiver exceptional sequence $\mb{E}$ completing $\mb{F}$ is given by $\{ \ep, \mb{F}_a', \mb{F}_a \}_{a=1,2,3}$ (the order of $a$ does not matter),
where $\mb{F}_a:=\e_{n-1}^a,\dots,\e_2^a,\e_1^a, \mb{F}_a':=\e_{2n-1}^a,\dots,\e_{n+2}^a,\e_{n+1}^a$ and $\ep = {^\alpha} (\e_{2n}-\e_n^1-\e_{n}^2-\e_n^3)$.
The quiver of $\mb{E}$ is the sextuple flag quiver, where $\e_i^a$ ($1\leq i\leq n-1$) and $\e_{n+i}^a$ ($1\leq i\leq n-1$) correspond to the vertices on the left and right three flags respectively, and $\ep$ corresponds to the central vertex $n$.
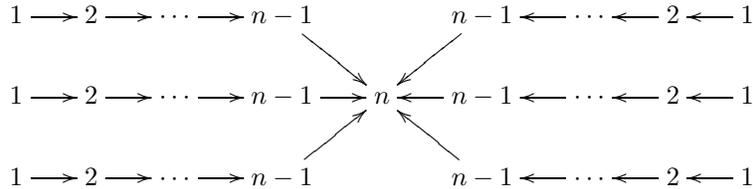
$$\sextupleflag$$

We call such a quiver sextuple flag quiver, denoted by $S_n^6$. Let $\beta_n^6$ be the dimension vector indicated as above.
We can verify by \eqref{eq:dimemb} that under the embedding the dimension vector $\beta_n^6$ goes to $\beta_{2n}$ of $T_{2n}$.
Let $\Delta_n^\triangledown$ be the ice quiver obtained from $\Delta_{2n}$ by deleting the vertices corresponding to $\mb{F}$.
\begin{figure}[h]
\begin{center}
$\hivesixp$
\end{center}
\caption{The ice quiver $\Delta_3^\triangledown$.}
\label{fig:hivep}
\end{figure}
It follows from Theorem \ref{T:invGUCA} that
\begin{proposition} The semi-invariant ring $\SI_{\beta_n^6}(S_n^6)$ is the GUCA $\uca(\Delta_n^\triangledown;\bs{\sigma}_n^\triangledown)$.
\end{proposition}
\noindent Here, $\bs{\sigma}_n^\triangledown$ can also be easily obtained from \eqref{eq:wtproj}.
Since it will appear in Proposition \ref{P:multipleLR}, we list it for readers' convenience.
Let $i = i_qn + i_r$ such that $0\leq i_r < n$, and similar for $j$ and $k=n-i-j$, then
\begin{align*} \bs{\sigma}_n^\triangledown(i,j) = (2-i_q-j_q-k_q)\e_n - \e_{i_r}^{1+3i_q} - \e_{j_r}^{2+3j_q} - \e_{k_r}^{3+3k_q}.
\end{align*}
\end{example}

\subsection{Vertex Removal}
Let $Q$ be a quiver, and $\beta$ a dimension vector of $Q$.
For a fixed vertex $r\in Q$, let $\br{Q}$ be the quiver such that $\br{Q}_0=Q_0\setminus \{r\}$,
and
$$\br{Q}_1=Q_1\setminus \{a\mid h(a)=r \text{ or } t(a)=r\} \cup \{[ab]\mid h(a)=t(b)=r\},$$
where $t([ab])=t(a)$ and $h([ab])=h(b)$.
The {\em vertex removal} at $r$ on $\Rep_\beta(Q)$ is the algebraic morphism $\Rep_\beta(Q) \to \Rep_{\br{\beta}}(\br{Q})$ sending $M$ to $\br{M}$ defined by
\begin{align*}
\br{M}(a) & = M(a) & \text{if  } & t(a),h(a)\neq r, \\
\br{M}([ab]) & = M(ab) & \text{if  } & t(a)\xrightarrow{a} r \xrightarrow{b} h(b).
\end{align*}

\begin{lemma} \label{L:projSI} Assume the following condition on $\beta$
\begin{equation} \label{eq:beta_r} \min\big( \sum_{h(a)=r} \beta(t(a)),\sum_{t(b)=r} \beta(h(b)) \big) \leq \beta(r).
\end{equation}
The vertex removal $\Rep_\beta(Q)\to \Rep_{\br{\beta}}(\br{Q})$
induces an embedding $\iota:\SI_{\br{\beta}}(\br{Q}) \hookrightarrow \SI_{\beta}(Q)$
mapping $\SI_{\br{\beta}}(\br{Q})$ onto $\bigoplus_{\sigma_0} \SI_\beta (Q)_{\sigma_0}$,
where $\sigma_0$ runs through all weights not supported on $r$.
\end{lemma}

\begin{proof} It suffices to show that $\SI_{\br{\beta}}(\br{Q})\cong \bigoplus_{\sigma_0} \SI_\beta (Q)_{\sigma_0}$.
Let $M_{\op{in}}:=\bigoplus_{h(a)=r} M(t(a))$ and $M_{\op{out}}:=\bigoplus_{t(b)=r} M(h(b))$.
By the first and second fundamental theorem of invariant theory of $\GL(V)$, we can identify
$$\Hom(M_{\op{in}},M(r))\times \Hom(M(r),M_{\op{out}}) /\!/ \GL(M(r))$$
with the rank at most $\beta(r)$ maps in $\Hom(M_{\op{in}}, M_{\op{out}})$.
But due to our assumption, all maps in $\Hom(M_{\op{in}}, M_{\op{out}})$ have rank at most $\beta(r)$.
By definition, $\bigoplus_{\sigma_0} \SI_\beta(Q)_{\sigma_0}=k[\Rep_\beta(Q)]^{\SL_{\br{\beta}}\times \GL(M(r))}$,
which is equal to $k[\Rep_{\br{\beta}}(\br{Q})]^{\SL_{\br{\beta}}}$.
\end{proof}

\noindent In other words, the $\sigma_0$ in the above lemma is in $\theta^\perp$, where $\theta$ is the weight corresponding to the unit vector $\e_r$.
It follows from Proposition \ref{P:outerproj} and Lemma \ref{L:projSI} that

\begin{theorem} \label{T:removal} Suppose that $\SI_\beta(Q)$ is a naturally graded UCA $\uca(\Delta;\bs{\sigma})$, and $B(\Delta)$ has full rank.
Let $r$ be a vertex of $Q$ satisfying \eqref{eq:beta_r}
and that $\bs{\sigma}(v)(r) = 0$ for all $v\in \Delta_0$ except for a set of frozen vertices $\b{v}$, where all $\bs{\sigma}(v)(r)<0$ (or all $>0$).
Then we have that $\SI_{\br{\beta}}(\br{Q})$ is the naturally graded UCA $\uca(\Deltauv;\bs{\sigma}(\hat{\b{v}}))$, where $\b{u}$ is attached to $\b{v}$.
\end{theorem}

Finally, we remark that Proposition \ref{P:outerproj} can also be used to remove arrows.
The outer automorphism group $\Aut_1(Q)$ of the quiver $Q$ is by definition $\prod_{i,j\in Q_0} \GL(A_{i,j})$,
where $A_{i,j}$ is the linear space spanned by arrows from $i$ to $j$.
In particular, for each arrow $a$ of $Q$, there is an associated torus $k_a^*$ in $\Aut_1(Q)$.
Since the action of $\Aut_1(Q)$ commutes with the action of $\GL_\beta$,
the torus $T_1:=\prod_{a\in Q_1} k_a^*$ induces another grading of the semi-invariant ring.
If the cluster in a seed is homogeneous with respect to the grading induced by $k_a^*$,
then we can use Proposition \ref{P:outerproj} to remove the arrow $a$.

\subsection{Partial Flag Varieties, Quintuple and Quadruple Flags, and a Triple Flag} \label{ss:ex_vr}

\begin{exercise}[Partial flag varieties] We are able to find the cluster structure of total coordinate ring of any partial flag varieties in type $A$. Since the divisor class group of a Grassmannian is $\mb{Z}$, its total coordinate ring is just the usual coordinate ring.
Hence we should be able to recover the result in \cite{Sc} as well.
We can start from the result in Example \ref{ex:cflag}, then remove some vertices on the long arm of $SF_n$ depending on the particular flag variety. The following hint may be useful.

Let $\b{r}_{l,m}= \{n-i\mid i=l,l+1,\dots,m\}$ for $n>m\geq l$.
We define a sequence of vertices $\b{u}_{l}:=\{(l-1,1),(l-2,2),\dots,(1,l-1)\}$.
This contains all mutable vertices of $l$-th row of $\Delta_n^\flat$.
If we want to remove all vertices in $\b{r}_{l,m}$, we first apply a sequence mutation in the following order
\begin{equation} \label{eq:mu_ml}
\b{u}_{l};\ \b{u}_{l+1},\b{u}_{l};\ \b{u}_{l+2},\b{u}_{l+1},\b{u}_{l};\ \dots\dots;\ \b{u}_{m},\b{u}_{m-1},\dots,\b{u}_{l}.
\end{equation}
Then we are able to apply Theorem \ref{T:removal}.
\end{exercise}

\begin{exercise}[Complete quintuple and quadruple flags] To get the cluster structure of complete quintuple flags, we start with Example \ref{ex:6flag}, then we remove all vertices in a flag in the order $n-1,n-2,\dots,1$.
Remove vertices in another flag, and we get a complete quadruple flag.
We shall find that

\begin{proposition} The semi-invariant ring $\SI_{\beta_n^5}(S_n^5)$ is the GUCA $\uca(\merge_n;\bs{\sigma}_n')$;
The semi-invariant ring $\SI_{\beta_n^4}(S_n^4)$ is the GUCA $\uca(\square_n;\bs{\sigma}_n'')$.
\end{proposition}
\noindent We are not going to give the precise construction of the quivers $\merge_n$ and $\square_n$ because it is rather clear.
When $n=3$, the quiver $\merge_3$ and $\square_3$ is displayed below.

\begin{figure}[h]
        \centering
        \begin{subfigure}[h]{0.3\textwidth}
                $\pentagonthree$
        \end{subfigure}                 {\quad \hfill}
        ~ 
        \begin{subfigure}[h]{0.3\textwidth}
                $\squarethree$
        \end{subfigure}                {\qquad\qquad \hfill} \\
        \caption{The quiver $\merge_3$ (left) and $\square_3$ (right)}\label{fig:flags}
\end{figure}
\end{exercise}

\begin{conjecture} The quiver $S_n^m$ can be projected from a triple flag quiver for any $m\geq 3, n\geq 2$.
The semi-invariant ring $\SI_{\beta_n^m}(S_n^m)$ of any complete $m$-tuple flag is an upper cluster algebra.
\end{conjecture}

What we will do in the rest of this subsection is to illustrate in a single example on
how to inductively apply the vertex removal technique (Theorem \ref{T:removal}) to find possible cluster structure.
If the method described below fails, we believe that the cluster structure may not exist.
However, the method is not deterministic.

\noindent {\bf The Method:}
Suppose that $\SI_\beta(Q)$ is a GUCA $\uca(\Delta;\bs{\sigma})$, and we want to remove a vertex $r$ of $Q$ satisfying \eqref{eq:beta_r}
Let us assume that $\beta$ is a Schur root, so the stabilizer in general position of the action is zero-dimensional.
Then the Krull dimension of $\SI_\beta(Q)$ is given by
$$\dim\Rep_\beta(Q)-\dim\SL_\beta=|Q_0|-\innerprod{\beta,\beta}.$$
We first compute the difference $d$ of $|Q_0|-\innerprod{\beta,\beta}$ before and after the vertex removal.
If Theorem \ref{T:removal} works, then $\bs{\sigma}(v)(r)$ should vanish for all but exactly $d$ frozen vertices,
where $\bs{\sigma}(v)(r)<0$ (or $>0$). These $d$ vertices are our choice of $\b{v}$.
Otherwise the method fails, or the cluster structure may not exist for $\SI_{\br{\beta}}(\br{Q})$.
Next we need to seek a suitable sequence of mutations such that after application $\bs{\sigma}(u)(r)=0$ for all mutable $u$.
If no such sequence exists, then the method fails, or the cluster structure may not exist.
If succeeds, we take $\Deltauv$, where $\b{u}$ is attached to $\b{v}$.

\begin{example} \label{Ex:T335}
In this example, we explain how to find a cluster algebra structure of the semi-invariant ring $\SI_\beta(T_{3,3,5})$,
where $(T_{3,3,5},\beta)$ is obtained from $(T_6,\beta_6)$ by removing the $5$'s of all three arms, and the $1$'s and $3$'s of the first two arms.
We also explain why the cluster algebra structure may not exist if we further forget either $1$ or $3$ of the third arm.
We suggest that readers play this example with B. Keller's applet \cite{Ke}.

We can directly apply Theorem \ref{T:removal} to remove all $5$'s and one $3$, say of the first arm.
The results correspond to the Figure \eqref{fig:rho5^1}, \eqref{fig:rho5^2}, \eqref{fig:rho5^3}, and \eqref{fig:rho3^1}.
We apply another mutation at $(3,2)$ to make the quiver easier to draw. This step is optional.

Next, to remove $1$ of the first arm, we first mutate at $(1,3)$ and $(1,2)$,
then we find that $\b{v}=\{(1,1),(1,4)\}$ and $\b{u}=\{(1,2)\}$ as required.
So we delete $\b{v}$ and freeze $\b{u}$, and perform an optional mutation at $(1,3)$.
Similarly, to remove $3$ of the second arm, we mutate at $(1,3)$ and $(2,3)$, then we find that $\b{v}=\{(0,3),(3,1)\}$ and $\b{u}=\{(2,3)\}$.
So we delete $\b{v}$ and freeze $\b{u}$, and perform an optional mutation at $(1,3)$.
Finally to remove $1$ of the first arm, we mutate at $(1,3)$ and $(2,1)$, then we find that $\b{v}=\{(1,2),(4,1)\}$ and $\b{u}=\{(2,1)\}$.
So we delete $\b{v}$ and freeze $\b{u}$, and perform an optional mutation at $(2,2)$.
The results correspond to the Figure \eqref{fig:rho1^1}, \eqref{fig:rho3^2}, and \eqref{fig:rho1^2}.
The interested readers can verify using \cite{Ke} that the types of the cluster algebras are affine $E_9,E_8,E_7,D_6,D_5,D_4,A_3$.

\begin{figure}[h]
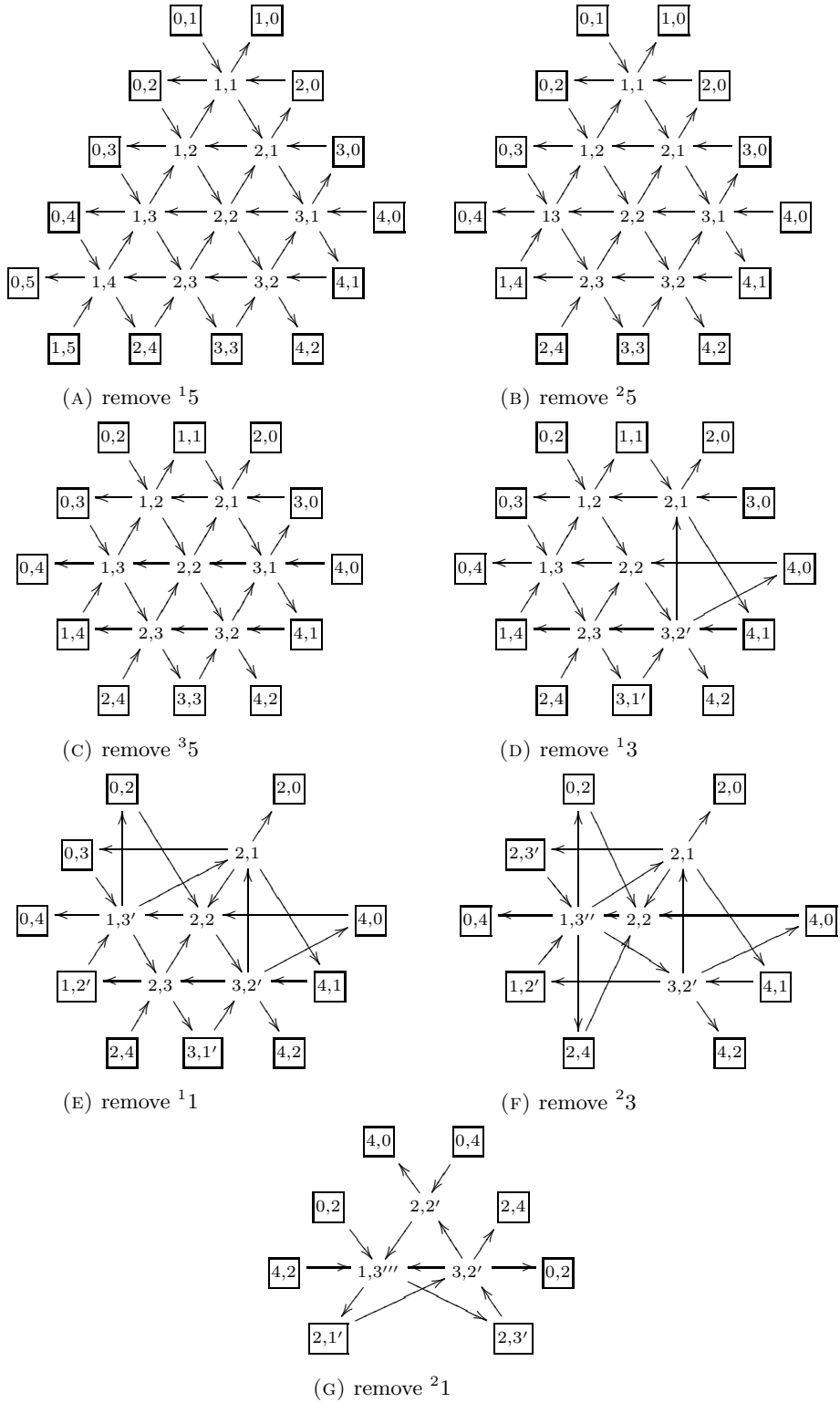

        \centering
        \begin{subfigure}[h]{0.3\textwidth}
                $\hivesixfive$
                \caption{remove $^15$}    \label{fig:rho5^1}
        \end{subfigure}                 {\hfill}
        ~ 
        \begin{subfigure}[h]{0.3\textwidth}
                $\hivesixfivefive$
                \caption{remove $^25$}    \label{fig:rho5^2}
        \end{subfigure}                 \hfill\hfill\hfill\\
        \begin{subfigure}[h]{0.3\textwidth}
                $\hivesixfivefivefive$
                \caption{remove $^35$}    \label{fig:rho5^3}
        \end{subfigure}                 {\hfill}
        \begin{subfigure}[h]{0.3\textwidth}
                $\hivesixfffthree$
                \caption{remove $^13$}    \label{fig:rho3^1}
        \end{subfigure}                 \hfill\hfill\hfill \\
        \begin{subfigure}[h]{0.3\textwidth}
                $\hivesixfffthreeone$
                \caption{remove $^11$}    \label{fig:rho1^1}
        \end{subfigure}                 {\hfill}
        \begin{subfigure}[h]{0.3\textwidth}
                $\hivesixfffttone$
                \caption{remove $^23$}    \label{fig:rho3^2}
        \end{subfigure}                 \hfill\hfill\hfill \\
        \begin{subfigure}[h]{0.3\textwidth}
                $\hivesixfffttoneone$
                \caption{remove $^21$}    \label{fig:rho1^2}
        \end{subfigure}  \qquad \qquad \qquad
        \caption{Operations towards $\SI_\beta(T_{3,3,5})$ in steps}\label{fig:animals}
\end{figure}

For the last one, the cluster variable associated to a vertex $v$ is semi-invariant of weight $\f_v$,
where $\f_{ij}$ as before and
\begin{align*}
\f_{21}' &= 3\e_6-2\e_1^3-\e_2^1-\e_2^2-\e_4^1-\e_4^2-\e_4^3,\\
\f_{23}' &= 3\e_6-\e_2^1-\e_2^2-2\e_3^3-\e_4^1-\e_4^2,\\
\f_{13}''' &= 2\e_6-\e_1^3-\e_2^1-\e_2^2-\e_3^3-\e_4^2,\\
\f_{32}' &= 2\e_6-\e_1^3-\e_2^1-\e_2^2-\e_3^3-\e_4^1,\\
\f_{22}' &= 2\e_6-\e_1^3-\e_3^3-\e_4^1-\e_4^2.
\end{align*}
One can verify that they all correspond to all real Schur roots.
The semi-invariant ring $\SI_\beta(T_{3,3,5})$ is this graded upper cluster algebra.

However, if we want to further remove $3$ or $1$ of the third arm, then we are unable to choose the set $\b{v}$ because the expected cardinality of $\b{v}$ is $2$. Even worse, no matter how we mutate from \eqref{fig:rho1^2},
there are at least $3$ vertices $v$ such that $\bs{\sigma}(v)(r)\neq 0$.
So we suspect that it may not be an upper cluster algebra.
\end{example}

\section{Cluster Characters from Quivers with Potentials} \label{S:CC}
\subsection{Quivers with Potentials}
In \cite{DWZ1} and \cite{DWZ2}, the mutation of quivers with potentials is invented to model the cluster algebras.
We emphasis here that ice quivers considered here can have arrows between frozen vertices.
Following \cite{DWZ1}, we define a potential $W$ on an ice quiver $\Delta$ as a (possibly infinite) linear combination of oriented cycles in $\Delta$.
More precisely, a {\em potential} is an element of the {\em trace space} $\Tr(\ckQ):=\ckQ/[\ckQ,\ckQ]$,
where $\ckQ$ is the completion of the path algebra $k\Delta$ and $[\ckQ,\ckQ]$ is the closure of the commutator subspace of $\ckQ$.
The pair $(\Delta,W)$ is an {\em ice quiver with potential}, or IQP for short.
For each arrow $a\in \Delta_1$, the {\em cyclic derivative} $\partial_a$ on $\widehat{k\Delta}$ is defined to be the linear extension of
$$\partial_a(a_1\cdots a_d)=\sum_{k=1}^{d}a^*(a_k)a_{k+1}\cdots a_da_1\cdots a_{k-1}.$$
For each potential $W$, its {\em Jacobian ideal} $\partial W$ is the (closed two-sided) ideal in $\ckQ$ generated by all $\partial_a W$.
The {\em Jacobian algebra} $J(\Delta,W)$ is the quotient algebra $\widehat{k\Delta}/\partial W$.
If $W$ is polynomial and $J(\Delta,W)$ is finite-dimensional, then the completion is unnecessary to define $J(\Delta,W)$.
This is the case throughout this paper.

The key notion introduced in \cite{DWZ1,DWZ2} is the {\em mutation} of quivers with potentials and their decorated representations.
For an ice quiver with {\em nondegenerate potential} (see \cite{DWZ1}), the mutation in certain sense ``lifts" the mutation in Definition \ref{D:Qmu}.
Since we do not need mutation in an explicit way, we refer readers to the original text.

\begin{definition} A {\em decorated representation} of a Jacobian algebra $J:=J(\Delta,W)$ is a pair $\mc{M}=(M,M^+)$,
where $M\in \Rep(J)$, and $M^+$ is a finite-dimensional $k^{\Delta_0}$-module.
\end{definition}

Let $\mc{R}ep(J)$ be the set of decorated representations of $J(\Delta,W)$ up to isomorphism. There is a bijection between two additive categories $\mc{R}ep(J)$ and $K^2(\proj J)$ mapping any representation $M$ to its minimal presentation in $\Rep(J)$, and the simple representation $S_u^+$ of $k^{\Delta_0}$ to $P_u\to 0$.
Suppose that $\mc{M}$ corresponds to a projective presentation
$P(\beta_1)\to P(\beta_0)$ for $\beta_1,\beta_0\in \mb{Z}_{\geqslant 0}^{Q_0}$.

\begin{definition} The {\em $\g$-vector} $\g(\mc{M})$ of a decorated representation $\mc{M}$ is the weight vector of its image in $K^b(\proj J)$, that is, $\g=\beta_1-\beta_0$.
The representation $\mc{M}$ is called {\em $\g$-coherent} if $\min(\beta_1(u),\beta_0(u))=0$ for all~vertices~$u$.
\end{definition}

\begin{definition} \label{D:ice}
For an IQP $(\Delta,W)$, {\em freezing} a set of vertices $\b{u}$ is the operation that
we set all vertices in $\b{u}$ to be frozen, but we do not delete arrows between frozen vertices.
We also keep the same potential $W$.
\end{definition}

\begin{remark} This definition is slightly different from \cite{Fs1} where we delete arrows between frozen vertices, and thus change potential as well. This difference can be ignored if we only look at the category of $\mu$-supported representations (see Definition \ref{D:mu_suprep}).
The reason is that for the two different definitions, the two categories are equivalent and their objects take the same value under the cluster character (Definition \ref{D:CC}).
However, if we insist on deleting arrows between frozen vertices, that will cause additional trouble in constructing the two functors in Section \ref{s:resind}.
\end{remark}

\begin{definition} \label{D:mu_suprep}
A decorated representation $\mc{M}$ is called {\em $\mu$-supported} if the supporting vertices of $M$ are all mutable.
We denote by $\mc{R}ep^{\g,\mu}(J)$ the set of all $\g$-coherent $\mu$-supported decorated representations of $J$.
\end{definition}

%
%

\subsection{The Cluster Character}
\begin{definition}[\cite{DWZ2}] \label{D:CC}
We define the {\em cluster character} $C: \mc{R}ep^{\g,\mu}(J) \to \mb{Z}(\b{x})$ by
\begin{equation} \label{eq:CC}
C(\mc{M})= \b{x}^{\g(\mc{M})} \sum_{\e} \chi\big(\Gr^{\e}(M) \big) \hat{\b{y}}^{\e},
\end{equation}
where $\Gr^{\e}(M)$ is the variety parameterizing $\e$-dimensional quotient representations of $M$, and $\chi(-)$ denotes the topological Euler-characteristic.
\end{definition}

\noindent It turns out \cite[Lemma 5.3]{Fs1} that if $\mc{M}\in \mc{R}ep^{\g,\mu}(J)$, then $C(\mc{M})$ is an element in $\uca(\Delta,\b{x})$.
We note that $C(\mc{M})$ has a well-defined $\g$-vector, which is $\g(\mc{M})$.

\begin{theorem}[{\cite[Theorem 5.6]{Fs1}}] \label{T:CC}
Suppose that IQP $(\Delta,W)$ is non-degenerate and $B(\Delta)$ has full rank.
Let $\mc{R}$ be a set of $\g$-coherent $\mu$-supported decorated representations with all distinct $\g$-vectors, then $C$ maps $\mc{R}$ (bijectively) to a set of linearly independent elements in the upper cluster algebra $\uca(\Delta)$.
\end{theorem}
\noindent However, it is not clear to us whether any UCA has a basis with distinct $\g$-vectors.

\begin{definition}[\cite{DF}]
To any $\g\in\mathbb{Z}^{\Delta_0}$ we associate the {\em reduced} presentation space $$\PHom_J(\g):=\Hom_J(P([\g]_+),P([-\g]_+)),$$
where $[\g]_+$ is the vector satisfying $[\g]_+(u) = \max(\g(u),0)$.
We denote by $\Coker(\g)$ the cokernel of a general presentation in $\PHom_J(\g)$.
\end{definition}
\noindent Reader should be aware that $\Coker(\g)$ is just a notation rather than a specific representation.
If we write $M=\Coker(\g)$, this simply means that we take a general presentation in $\PHom_J(\g)$, then let $M$ to be its cokernel.

\begin{definition}[{\cite{Fs1}}] \label{D:mu_supg}
A weight vector $\g\in K_0(\proj J)$ is called {\em $\mu$-supported} if $\Coker(\g)$ is $\mu$-supported.
Let $G(\Delta,W)$ be the set of all $\mu$-supported vectors in $K_0(\proj J)$.
\end{definition}

\begin{definition}[\cite{P}]
We define the {\em generic character} $C_W:G(\Delta,W)\to \mb{Z}(\b{x})$~by
\begin{equation} \label{eq:genCC}
C_W(\g)=\b{x}^{\g} \sum_{\e} \chi\big(\Gr^{\e}(\Coker(\g)) \big) {\b{y}}^{\e}.
\end{equation}
\end{definition}

\noindent It is known \cite[Lemma 5.11]{Fs1} that a general presentation in $\PHom_J(\g)$ corresponds to a $\g$-coherent decorated representation.
So the image of $C_W$ is also contained in $\uca(\Delta,\bs{x})$. In particular, we have that $G(\Delta,W)\subseteq G(\Delta)$.

\begin{remark} \label{R:basis} Assume that we are in the situation of Theorem \ref{T:CC}.
Moreover, we assume that $\uca(\Delta)$ is $\bs{\sigma}$-graded with finite-dimensional graded components.
Let $G(\mc{R})$ be the set of all $\g$-vectors in $\mc{R}$.
If $C(\mc{R})$ forms a basis of $\uca(\Delta)$, then so does $C_W(G(\mc{R}))$ by Lemma \ref{L:independent}.
\end{remark}


\begin{definition} \label{D:model}
We say that an IQP $(\Delta,W)$ {\em models} an algebra $\mc{A}$ if the generic cluster character maps $G(\Delta,W)$ onto a basis of $\mc{A}$.
If $\mc{A}$ is the upper cluster algebra $\uca(\Delta)$, then we simply say that $(\Delta,W)$ is a {\em cluster model}.
If in addition $G(\Delta,W)$ is given by lattice points in some polyhedron, then we say that the model is {\em polyhedral}.
\end{definition}

\noindent In this definition we do not require the IQP to be non-degenerate.
If $\uca(\Delta)$ is $\bs{\sigma}$-graded with finite-dimensional graded components, then by Remark \ref{R:basis} $(\Delta,W)$ being a cluster model is equivalent to that $\uca(\Delta)$ is $\g$-indexed and $G(\Delta,W)=G(\Delta)$.


\section{The IQP Models} \label{S:model}
\subsection{Restriction and Induction} \label{s:resind}

\begin{definition} For an IQP $(\Delta,W)$, {\em deleting} a set of vertices $\b{v}\subset \Delta_0$ is the operation that
we delete all vertices in $\b{v}$ for $\Delta$, and set all incoming and outgoing arrows of $v\in \b{v}$ to be zero in $W$.
We denote the new IQP by $(\Delta_\b{v},W_{\b{v}})$.
\end{definition}

\noindent The operation of deleting $\b{v}$ is the same as {\em restricting} $(\Delta,W)$ to the vertex set $\Delta_0\setminus \b{v}$ as in \cite[Definition 8.8]{DWZ1}.
In what follows, we shall mainly consider the case when $\b{v}$ is a single vertex $e$. The general case can be easily treated by induction.
According to \cite[Proposition 8.9]{DWZ1}, the restriction induces an algebra epimorphism $p:J(\Delta,W)\twoheadrightarrow J(\Delta_e,W_e)$.

We denote $J_e:=J(\Delta_e,W_e)$. Let $\res_e:\Rep(J_e)\to \Rep(J)$ be the restriction of scalar functor induced by $p$.
Concretely, if $N\in\Rep(J_e)$, then the extension of $N$ by zeros to the quiver $\Delta$ is clearly a representation of $J$.
Such a representation is $\res_e(N)$. Conversely, we denote the functor $\otimes_J J_e:\Rep(J)\to \Rep(J_e)$ by $\ind_e$.
It is well-known \cite{Sk} that the functor $\ind_e$ is left adjoint to the exact functor $\res_e$.

We extend the functor $\res_e$ to decorated representations by copying the decorated spaces.
We denote such a map still by $\res_e: \mc{R}ep(J_e)\to \mc{R}ep(J)$.
The functor $\ind_e$ can be naturally extended.
For $P_1\xrightarrow{f} P_0 \in \proj J$, we define $\ind_e(f) \in \proj J_e$ as $P_1\otimes_J J_e \xrightarrow{f\otimes_J J_e} P_0\otimes_J J_e$.
We note that $P_e\otimes_J J_e=0$ but $P_u \otimes_J J_e$ is the indecomposable projective $P_u$ in $\Rep(J_e)$.
So $\g(\ind_e(\mc{M}))=\ll_e(\g(\mc{M}))$, where $\ll_e$ is the linear map forgetting the $e$-th coordinate.

\begin{lemma} \label{L:gdiff} The $\g$-vector of $\mc{M}$ and $\res_e(\mc{M})$ only differ on the $e$-th coordinate.
Moreover, $\g(\res_e(\mc{M}))(e)$ is always non-negative.
\end{lemma}

\begin{proof} Recall from \cite[(1.13)]{DWZ1} that \begin{equation}
\label{eq:delta2dim} \g(u)=\dim(\Coker \gamma_u)-\dim M(u) + \dim M^+(u). \end{equation}
We refer readers to \cite[Section 10]{DWZ1} for the definition of the map $\gamma_u$.
Comparing $\mc{M}$ and $\res_e(\mc{M})$, we find that there is no changes on the three terms on the right for $u\neq e$.
The last statement is also clear from the equation.
\end{proof}

\begin{lemma} \label{L:resind} The composition $\ind_e\cdot \res_e$ is the identity functor.
In particular, if $e$ is frozen and $\Coker(f)$ is $\mu$-supported, then so is $\Coker(\ind_e(f))$.
\end{lemma}

\begin{proof} Since the composition is right exact, it suffices to show that it is an identity on the indecomposable projective representations.
By Lemma \ref{L:gdiff}, for any $P_u \in \Rep(J_e)$ we have the following presentation
$m P_e\to P_u\to \res_e(P_u)\to 0$ for some $m\in\mb{Z}_{\geqslant 0}$.
We apply the right exact functor $\ind_e$ to the exact sequence, and get
$$0=m \ind_e(P_e)\to P_u=\ind_e(P_u)\to \ind_e \res_e(P_u)\to 0.$$
If $e$ is frozen, then $f=\res_e(f')$ for some $\mu$-supported $f'$. The last statement follows.
\end{proof}

Suppose that $P_0=\bigoplus_{i=1}^m P_{u_i}$ and $P_1=\bigoplus_{j=1}^n P_{v_j}$.
Then we can view $P_0,P_1$ as row vectors with entries in $P_{u_i},P_{v_j}$, and $f:P_1\to P_0$ can be represented by an $n\times m$ matrix $(a_{ji})$ with $a_{ji}$ a linear combination of paths from $u_i$ to $v_j$.
From this description, we see that the functor $\otimes_J J_e$ induces a surjective algebraic map $\Hom_J(P_1,P_0)\to \Hom_{J_e}(P_1\otimes_{J}J_e,P_0\otimes_{J}J_e)$.
In particular, if $f$ is general enough in $\Hom_J(P_1,P_0)$, then
both $C(\mc{M})$ and $C(\ind_e(\mc{M}))$ take the value of the generic character, where $\mc{M}$ is the decorated representation corresponding to $f$.

\subsection{Orthogonal Projection}
We first focus on the case when $e$ is frozen.
We keep our assumption that $\ep=\bs{\sigma}(e)$ is real.

\begin{lemma} \label{L:commute_e} We have the following commutative diagram for $e$ frozen.
$$\vcenter{\xymatrix@C=5ex{
\mc{R}ep^{\g,\mu}(J) \ar@{^(->}[r]^{C} \ar@{->>}[d]_{\ind_e} & \uca(\Delta,\b{x};\bs{\sigma}) \ar@{->>}[d]_{\ell_e} \\
\mc{R}ep^{\g,\mu}(J_e) \ar@{^(->}[r]^{C\quad} & \uca(\Delta_e,\b{x}_e;\bs{\sigma}_e)  
}}$$
\end{lemma}

\begin{proof} The surjectivity of $\ell_e$ follows from Proposition \ref{P:subalgebra}.
For $\mc{M}\in \mc{R}ep^{\g,\mu}(J)$, we have that $\mc{M}= \res_e(\mc{N})$ for some $\mc{N} \in \mc{R}ep^{\g,\mu}(J_e)$ because $e$ is frozen. Recall from Lemma \ref{L:resind} that $\ind_e\cdot \res_e(\mc{N}) = \mc{N}$.
By the definition of $\res_e$, $\mc{N}$ is isomorphic to $\res_e(\mc{N})$ restricting to the subquiver $\Delta_e$.
By definition $C(\mc{M}) = \b{x}^{\g(\mc{M})} \sum \chi(\Gr^{\e}(\mc{M})) \b{y}^{\e}$.
By Lemma \ref{L:projB} we have that
$$\ell_e(C(\mc{M})) = {\ell}_e(\b{x}^{\g(\mc{M})}) \sum \chi(\Gr^{\e}(\mc{M})) (\b{y}')^{\e}.$$
On the other hand, $$C(\ind_e(\mc{M})) = C(\mc{N}) =\b{x}_e^{\g(\mc{N})} \sum \chi(\Gr^{\e}(\mc{M})) (\b{y}')^{\e}.$$
Since ${\ell}_e(\b{x}^{\g(\mc{M})}) = \b{x}_e^{\ll_e(\g(\mc{M}))}= \b{x}_e^{\g(\mc{N})}$, we conclude that
$\ell_e C=C\ind_e$.

\end{proof}

\begin{remark} If $e$ is a projector, then we have the embedding $i_e: \uca(\Delta_e;\bs{\sigma}_e) \hookrightarrow \uca(\Delta;\bs{\sigma})$.
However, we want to warn readers that another functor $\res_e$ and the embedding do not fit into the above commutative diagram in general, i.e., $C \res_e \neq i_eC$.
\end{remark}

\begin{theorem} \label{T:frozen_model}  Suppose that $e$ is a projector in $\uca(\Delta;\bs{\sigma})$.
If $(\Delta,W)$ models $\uca(\Delta;\bs{\sigma})$, then $(\Delta_e,W_e)$ models $\uca(\Delta_e;\bs{\sigma}_e)$.
\end{theorem}

\begin{proof} We first assume that $e$ is frozen. We first show that $C_{W_e}(G(\Delta_e,W_e))$ forms a basis of $\uca(\Delta_e)$.
By assumption, there is a subset $\mc{R}$ of $\mc{R}ep^{\g,\mu}(J(\Delta,W))$ with distinct $\g$-vectors in $\Ker \ep_l$ (defined after Remark \ref{R:cluster})
such that $C(\mc{R})$ forms a basis of $\uca(\Delta;\bs{\sigma})_{\ep^\perp}$.
Moreover, by the remark in the end of Section \ref{s:resind},
each element in $\mc{R}$ can be chosen to take the value of the generic character.
Since $e$ is a projector, $\uca(\Delta_e;\bs{\sigma}_e)=\uca(\Delta;\bs{\sigma})_{\ep^\perp}$ by Proposition \ref{P:subalgebra}.
So $\ell_e (C(\mc{R}))=C(\ind_e(\mc{R}))$ forms a basis of $\uca(\Delta_e;\bs{\sigma}_e)$.
Since $\g(\ind_e(\mc{M}))=\ll_e(\g(\mc{M}))$,
by a similar argument as in Corollary \ref{C:Gproj_e} we see that
elements in $\ind_e(\mc{R})$ have distinct $\g$-vectors.

Now we come back to the general situation, that is, $e$ is not necessary frozen.
We first claim that $\ell_e(C_W(G(\Delta^e,W)))$ spans $\uca(\Delta_e;\bs{\sigma}_e)$ if $C_W(G(\Delta,W))$ spans $\uca(\Delta;\bs{\sigma})$.
For any such $\g\in G(\Delta,W)$, we can always find some $m$ sufficiently large such that $\Coker(\g+m\e_e)$ is not supported on $e$, i.e., $\g+m\e_e\in G(\Delta^e,W)$.
By the definition of $\ell_e$, the $\g$-vector of $\ell_e(C_W(\g))$ is equal to $\g-\innerprod{\ep,\g\bs{\sigma}} \e_e$.
Note that this is also the $\g$-vector of $\ell_e(C_W(\g+m\e_e))$.
Since $\ell_e(C_W(G(\Delta,W)))$ spans $\uca(\Delta_e,\b{x}_e;\bs{\sigma}_e)$, so is $\ell_e(C_W(G(\Delta^e,W)))$ by Lemma \ref{L:independent}.
Finally, we can apply the previous conclusion to $(\Delta^e,W)$.
\end{proof}

\begin{remark} If the model of $\uca(\Delta)$ is polyhedral, then so is the model of $\uca(\Delta_e)$ because $\ll_e$ is totally unimodular.
The similar remark also applies to Theorem \ref{T:hat}.
\end{remark}

\begin{corollary} \label{C:model_proj} Suppose that $(\Delta,W)$ models $\SI_{\beta}(Q)$ and $B(\Delta)$ has full rank.
If $\bs{\sigma}(e)={^\sigma}\ep$ is a real weight extremal in $\mb{R}^+\Sigma_\beta(Q)$, then $(\Delta_e,W_e)$ models $\SI_{\beta_\ep}(Q_\ep)$.
\end{corollary}

\begin{proof} This follows from Theorem \ref{T:invGUCA}, Proposition \ref{P:extremal_proj}, and Theorem \ref{T:frozen_model}.
\end{proof}

\subsection{Example continued}
The quiver $\Delta_n$ can be equipped with a {\em rigid} potential $W_n$ (see \cite{Fs1} for detail).
It is proved in \cite{Fs1} that $(\Delta_n,W_n)$ models the semi-invariant ring $\SI_{\beta_n}(T_n)$.
Moreover, $G(\Delta_n,W_n)$ is given by the lattice points in the hive cone ${\sf G}_n$ defined as follows.
A maximal straight path in $\Delta_n$ must start from a frozen vertex and end in another frozen vertex, eg., from $(i,0)$ to $(0,i)$.
We define ${\sf G}_n$ by \begin{equation} \label{eq:cone_sp}
\text{$\sum_{u\in p} \g(u) \geq 0$ for
all strict subpaths $p$ of a maximal straight path.}  \end{equation}
Here, subpaths include all trivial paths $e_v$ for $v$ frozen, and the notation $\sum_{u\in p} \g(u)$ stand for
$\g(u_0)+\g(u_1)+\cdots+ \g(u_t)$ if $p$ is the path $u_0 \xrightarrow{} u_1 \xrightarrow{}\cdots\xrightarrow{} u_t$.

\begin{definition} Given a weight configuration $\bs{\sigma}$ of a quiver $\Delta$ and a convex polyhedral cone ${\sf G} \subset \mb{R}^{\Delta_0}$,
we define the (not necessarily bounded) convex polytope ${\sf G}(\sigma)$ as ${\sf G}$ cut out by the hyperplane sections $\g \bs{\sigma} = \sigma$.
\end{definition}
\noindent It follows from Knutson-Tao \cite{KT} and Derksen-Weyman \cite{DW1} (see \cite{Fs1}) that the lattice points in ${\sf G}_n(\sigma)$ counts the Littlewood-Richardson coefficient $c_{\mu(\sigma),\nu(\sigma)}^{\lambda(\sigma)}$.
We refer readers to \cite[7.1]{DW2} for the definition of the partitions $\lambda(\sigma),\mu(\sigma)$, and $\nu(\sigma)$.

\begin{example}[Example \ref{ex:cflag} continued]
By Corollary \ref{C:model_proj}, $G(\Delta_n^{\flat};\bs{\sigma}_n^{\flat})$ is given by lattice points in some polyhedral cone
denoted by ${\sf{G}}_n^\flat$.
Similar to the calculation in \cite[Section 6]{Fs1}, we find that the cone ${\sf{G}}_n^\flat$ can also be described by \eqref{eq:cone_sp}.

By a calculation similar to \cite[Proposition 1]{DW1}, we find that
the dimension of each graded piece $\SI_{\beta_n'}(SF_n)_\sigma$ counts the {\em Kostka number} $K_{\lambda}^{\mu}$,
which is by definition the multiplicity of the irreducible complex representation $\S_\lambda$ of $\GL_n$
in the tensor product $\S_{\mu(1)}\otimes \S_{\mu(2)}\otimes \cdots\otimes \S_{\mu(n)}$.
The partitions $\lambda$ and the weight $\mu$ are determined by $\sigma$ as follows
$$\lambda(\sigma)=((n-1)^{-\sigma(n-1)},\dots,2^{-\sigma(2)},1^{-\sigma(1)})^*,\ \mu(\sigma)=(\sigma(1'),\sigma(2'),\dots,\sigma(n')).$$
Alternatively, one can see this from the identification in Theorem \ref{T:SIemb}.
Under the identification $\SI_{\beta_n'}(SF_n)_{\sigma} \cong \SI_{\beta_n}(T_n)_{\iota_{\mb{F}}(\sigma)}$, the Kostka coefficient can be interpreted as certain Littlewood-Richardson coefficient.
This interpretation is well-known.

\begin{proposition} \label{P:Kostka} The Kostka number $K_{\lambda(\sigma)}^{\mu(\sigma)}$ is counted by the lattice points in ${\sf G}_n^\flat(\sigma)$.
\end{proposition}
\end{example}

\begin{example}[Example \ref{ex:ucell} continued]
By Corollary \ref{C:model_proj}, $G(\Delta_n^{\flat\flat};\bs{\sigma}_n^{\flat\flat})$ is given by lattice points in some polyhedral cone denoted by ${\sf{G}}_n^{\flat\flat}$.
Since the semi-invariant ring is just a polynomial ring in $N=\binom{n}{2}$ variables, we see that the cone is generated by $N$ extremal rays.
Moreover, the cone ${\sf{G}}_n^{\flat\flat}$ can also be described by \eqref{eq:cone_sp}, and we denote this hyperplane representation by $\g H_n \geq 0$.
The cone is related to {\em Kostant's vector partition function} $k_A$ \cite[24.1]{Hu} in type $A$ as follows.
Recall that $k_A(\lambda)$ counts the lattice points in the polytope $${\sf K}_A(\lambda):=\{ \h\in \mb{Z}_{\geqslant 0}^N \mid \h\Phi_n=\lambda \},$$
where $\Phi_n$ is the matrix form by rows of positive roots in type $A_{n-1}$.
It is easy to verify by linear algebra that $H_n\Phi_n=\bs{\sigma}_n^{\flat\flat}$, so we have that ${\sf G}_n^{\flat\flat}(\sigma) = {\sf K}_A(\lambda(\sigma))$,
where $\lambda(\sigma)=\g H_n^{-1}\bs{\sigma}_n^{\flat\flat}$ if $\sigma=\g\bs{\sigma}_n^{\flat\flat}$.
Since $H_n$ is {\em totally unimodular}, the two polytope have the same number of lattice points.

\begin{question} Recall that Kostant's formulae \cite[24.2]{Hu}
$$K_{\lambda}^{\mu} = \sum_{\omega\in S_n} (-1)^{\vep(\omega)} k_A(\omega\cdot\lambda-\mu)$$
relates any Kostka number to an alternating sum of values in the partition function, and Steinberg's formulae \cite[24.4]{Hu}
$$c_{\mu,\nu}^\lambda = \sum_{\omega\in S_n} (-1)^{\vep(\omega)} K_{\nu}^{\omega\cdot\lambda-\mu}$$
relates any Littlewood-Richardson coefficient to an alternating sum of Kostka numbers.
Here $\omega\cdot\lambda$ is the dot action of $S_n$ defined by
$\omega\cdot\lambda = \omega(\lambda+\rho)-\rho$ for $\rho$ the half sum of positive roots in type $A$.
Can we see these two formula from the relations among polytopes ${\sf G}_n^{\flat\flat},{\sf G}_n^{\flat}$, and ${\sf G}_n$?
\end{question}
\end{example}

\begin{example}[Example \ref{ex:6flag} continued]
By Corollary \ref{C:model_proj}, $G(\Delta_n^\triangledown;\bs{\sigma}_n^\triangledown)$ is given by lattice points in some polyhedral cone denoted by ${\sf G}_n^\triangledown$.
By a similar calculation in \cite[Section 6]{Fs1}, we find that the cone ${\sf{G}}_n^\triangledown$ can also be described by \eqref{eq:cone_sp}.
Similar to \cite[Proposition 1]{DW1}, the dimension of each graded piece $\SI_{\beta_n^6}(S_n^6)$ counts certain {\em multiple} Littlewood-Richardson coefficients $c_{\bs{\eta}}^\lambda$ (\cite[Section 1]{Fk1}), where $\bs{\eta}=(\eta_1,\eta_2,\dots,\eta_5)$ and $\lambda$ are determined by the weight $\sigma$ similar to \cite[7.1]{DW2}.

\begin{proposition} \label{P:multipleLR} The multiple Littlewood-Richardson coefficients $c_{\bs{\eta}(\sigma)}^{\lambda(\sigma)}$ is an ordinary Littlewood-Richardson coefficient $c_{\mu',\nu'}^{\lambda'}$, which is counted by the lattice points in ${\sf G}_n^\triangledown(\sigma)$.
\end{proposition}

\end{example}

\subsection{Vertex Removal} \label{S:VR}

Suppose that $\b{u}$ is attached to $\b{v}$, and $(\Delta,\b{x};\bs{\sigma})$ is a graded seed.
If we restrict the seed to $\Deltauv$, we get a new graded seed $(\Deltauv,\b{x}(\hat{\b{v}});\bs{\sigma}(\hat{\b{v}}))$ as in Section \ref{s:notin}.
Denote $J_{\b{v}}:=J(\Delta_{\b{v}},W_{\b{v}})$.
Recall that in the situation of Proposition \ref{P:outerproj}, we have an epimorphism $p_{\b{v}}: \uca(\Delta;\bs{\sigma}) \twoheadrightarrow \uca(\Deltauv;\bs{\sigma}(\hat{\b{v}}))$.
The proof of the following lemma, theorem, and corollary is completely analogous to (but easier than)
that of Lemma \ref{L:commute_e}, Theorem \ref{T:frozen_model}, and Corollary \ref{C:model_proj},
so we leave the details to the readers.

\begin{lemma} We have the following commutative diagram under the assumption of Proposition \ref{P:outerproj}.
$$\vcenter{\xymatrix@C=5ex{
\mc{R}ep^{\g,\mu}(J) \ar@{^(->}[r]^{C} \ar@{->>}[d]_{\res_{\b{v}}} & \uca(\Delta,\b{x};\bs{\sigma}) \ar@{->>}[d]_{p_{\b{v}}} \\
\mc{R}ep^{\g,\mu}(J_{\b{v}}) \ar@{^(->}[r]^{C\quad} & \uca(\Deltauv,\b{x}(\hat{\b{v}});\bs{\sigma}(\hat{\b{v}}))  
}}$$
\end{lemma}

\begin{theorem}  \label{T:hat} Under the same assumption as Proposition \ref{P:outerproj},
we suppose that $(\Delta,W)$ models $\uca(\Delta;\bs{\sigma})$.
Then $(\Delta_{\b{v}},W_{\b{v}})$ models $\uca(\Deltauv;\bs{\sigma}(\hat{\b{v}}))$.
\end{theorem}

\begin{corollary} \label{C:removal_model} Under the same assumption as Theorem \ref{T:removal},
we suppose that the semi-invariant ring $\SI_\beta(Q)$ is modeled by $(\Delta,W)$.
Then $\SI_{\br{\beta}}(\br{Q})$ is modeled by $(\Deltauv,W_{\b{v}})$.
\end{corollary}
\noindent We are not going provide more examples for this part. Good examples for illustration are contained in \cite[Section 10]{Fk1}.

\bibliographystyle{amsplain}

\end{document}